\documentclass[12pt, a4paper]{amsart}
\usepackage{amsmath,amssymb,amscd,amsfonts}
\newtheorem{theorem}{Theorem}[section]
\newtheorem{rem}[theorem]{Remark}
\newtheorem{defn}[theorem]{Definition}
\newtheorem{question}[theorem]{Question}
\newtheorem{ex}[theorem]{Example}
\newtheorem{lemma}[theorem]{Lemma}

\newtheorem{proposition}[theorem]{Proposition}

\newtheorem{corollary}[theorem]{Corollary}

\newcommand\ZZ{{\mathbb{Z}}}

\def\Ram{\mathop{\rm Ram}\nolimits}
\def\Chow{\mathop{\rm Chow}\nolimits}

\def\Mov{\mathop{\rm Mov}\nolimits}
\def\rk{\mathop{\rm rk}\nolimits}
\def\det{\mathop{\rm det}\nolimits}

\def\Ker{\mathop{\rm Ker}\nolimits}

\def\codim{\mathop{\rm codim}\nolimits}

\def\Sym{\mathop{\rm Sym}\nolimits}
\def\Hom{\mathop{\rm Hom}\nolimits}

\def\Chow{\mathop{\rm Chow}\nolimits}

\def\cO{{\mathcal O}}

\def\cJ{{\mathcal J}}
\def\cE{{\mathcal E}}

\def\cG{{\mathcal G}}
\def\cH{{\mathcal H}}
\def\cF{{\mathcal F}}
\def\cL{{\mathcal L}}
\def\cQ{{\mathcal Q}}

\def\cV{{\mathcal V}}

\let\ol=\overline
\let\wt=\widetilde
\let\wh=\widehat
\def\bQ{{\mathbb Q}}
\def\bC{{\mathbb C}}
\def\bR{{\mathbb R}}
\def\bZ{{\mathbb Z}}
\def\bP{{\mathbb P}}

\begin{document}

\title[]{Foliations with positive slopes\\
and birational stability of orbifold\\
 cotangent bundles}

\author{Fr\'ed\'eric Campana, Mihai P\u aun}

\address{Fr\'ed\'eric Campana\\
Universit\'e Lorraine \\
 Institut Elie Cartan\\
Nancy\\
KIAS\\
85 Hoegiro, Dongdaemun-gu\\
Seoul 130-722, South Korea\\
Institut Universitaire de France\\}
\email{frederic.campana@univ-lorraine.fr}

\address{Mihai P\u aun\endgraf
KIAS\endgraf
85 Hoegiro, Dongdaemun-gu\endgraf
Seoul 130-722, South Korea,\endgraf
and\endgraf
Department of Mathematics\endgraf
College of Liberal Arts and Sciences\endgraf 
University of Illinois at Chicago\endgraf 
851 S. Morgan Street\endgraf
Chicago IL 60607-7045, US\endgraf}
\email{mpaun@uic.edu}
\date{\today}

\date{\today}

\begin{abstract} 
Let $X$ be a smooth connected projective manifold, together with an snc orbifold
divisor $\Delta$, such that the pair $(X, \Delta)$ is log-canonical. 
 If $K_X+\Delta$ is not pseudo-effective, we show, among other things, that any quotient of its orbifold cotangent 
 bundle has a pseudo-effective determinant. This improves considerably our
previous result \cite{CP13}, where generic positivity instead of pseudo-effectivity was obtained. One of the  new ingredients in the proof is a version of the Bogomolov-McQuillan algebraicity criterion for holomorphic foliations whose minimal slope with respect to a movable class (instead of an ample complete intersection class) is positive.
\end{abstract}

\maketitle

\tableofcontents

\section{Introduction}

In the present text we evaluate the \emph{positive directions} of the tangent bundle of a projective manifold by means of the slope of its  subsheaves with respect to classes of movable curves. The crucial property is the birational nature of this notion. We show that the positive directions of the tangent bundle are given by fibering-type contractions birationally preserved. They are the same ones as those appearing in the Log-minimal model program, which rest on much more delicate notions and arguments, in some sense dual to those presented here.

Our basic tool and starting point is Theorem \ref{rcintro} below. This result is also valid in the orbifold context, considerably extending its range of applicability. The arguments we use 
lead to positivity/negativity properties of tensor powers of orbifold cotangent bundles. In this article, we work in characteristic zero exclusively. Classical results on rational curves are known, and thus quoted here, only when the orbifold divisor is zero.

\noindent Let $X$ be a projective manifold, and let $\cF\subset TX$ be a holomorphic foliation. Given a movable class $\alpha\in H^{n-1,n-1}(X,\bR)$ on $X$, the condition: $\mu_{\alpha, min}(\cF)>0$ means that the inequality of intersection numbers:
\begin{equation}\label{0009}
c_1(\cQ).\alpha > 0,
\end{equation}
holds for any non-zero quotient $\cF\to \cQ\to 0$.
\smallskip

\noindent Our first main result is: 

\begin{theorem}\label{rcintro} Let $X$ be projective smooth, and let $\cF\subset TX$ be a foliation
such that $\mu_{\alpha,min}(\cF)>0$ for some movable class $\alpha$.
Then $\cF$ is an algebraic foliation and the closure of its leaves are rationally connected.
\end{theorem} 
\noindent 
The algebraicity statement is the \emph{movable version} of the Bogomolov-McQuillan
algebraicity criterion \cite{BMQ}, where the class $\alpha$ is a complete intersection class $[C]=[H]^{n-1}$
of (very) ample hypersurfaces. The condition \eqref{0009} in this case means that the restriction $\cF|_{C}$ is ample. The proof given here follows the ideas from \cite{BMQ}, strengthened by the theory of semi-stability with respect to movable classes introduced and developed in \cite{CPe}. A main difference with \cite{BMQ} is that we do not restrict to movable curves of the given class. The failure of Mehta-Ramanathan in this context would anyhow prevent from doing this.

The rational connectedness statement is obtained by a simple and direct combination of several results: the existence of a `relative rational quotient' for any fibration, the pseudo-effectivity of the canonical bundle of its base by \cite{GHS}, and Theorem \ref{pseffff} below, asserting (in a more refined version) the pseudo-effectivity of the relative canonical bundle of a fibration having generic fibres with pseudo-effective canonical bundle. The slope considerations are central in this proof, as well as their birational preservation in the case of movable classes.

%\footnote{For simplification, we skip here some technical details (see \S4.2), related to birational modifications, and only give the idea of the proof.}by contradiction. Indeed: $\cF$ is then algebraic, and is so the relative tangent sheaf of a rational dominant fibration $f:X\dasharrow Z$. Any such $f$ has a relative `rational quotient' fibration $r:X\dasharrow Y, s:Y\dasharrow Z$, with $f=s\circ r$, the generic fibre of $r$ (resp. $s$) being rationally connected (resp. having a pseudo-effective canonical bundle, by \cite{GHS} and \cite{BDPP}). Assuming that the generic fibre of $f$ is not rationally connected means that $\dim(Y)> \dim(Z)$. By Theorem \ref{pseffff} below, $K_{Y/Z}=-\det(\cH)$ is pseudo-effective\footnote{Additional terms involving multiple fibres must be taken into account.}, with $\cH:=Ker(ds)$. But $r^*(\cH)$ is a quotient of $\cF$, and thus $\mu_{\alpha}(r^*(\cH))=\mu_{r_*(\alpha)}(\cH)>0$. This contradicts the pseudo-effectivity of $K_{Y/Z}$, since $r_*(\alpha)$ is movable on $Z$.

This proof radically differs from the previous ones given in \cite{BMQ} and \cite{KST} in the special case of $\alpha$ a complete intersection class.
\medskip

\noindent Theorem \ref{rcintro} plays a crucial role in the proof of the next statement (labeled Theorem \ref{quotfol} in section 4).
\smallskip

\begin{theorem}\label{capet} Let $X$ be a non-singular projective manifold, and let $\cF$ be a foliation on $X$, with $K_{\cF}$ pseudo-effective. For any $m>0$, the determinant bundle of any coherent, torsion-free quotient of $\displaystyle \otimes ^m\Omega^1_{\cF}$
is pseudo-effective.
 \end{theorem}

\noindent When $K_X$ is pseudo-effective, Theorem \ref{capet} follows directly from Theorem \ref{main} below, since $\Omega^1_{\cF}$ is a quotient of $\Omega^1_X$. Remark however that it is interesting in its own right, since foliations $\cF$ with $K_{\cF}$ pseudo-effective on some uniruled $X's$ do exist (cf. \cite{CP13}). In fact, Theorem \ref{capet} does not seem provable with the methods of \cite{CP13}, even using also \cite{BCHM}. 
\medskip
 
\medskip

\noindent Our original motivation was to establish the birational stability of the cotangent bundle
$\Omega^1{(X, \Delta)}$ of smooth log-canonical orbifold pairs $(X, \Delta)$ for which $K_X+ \Delta$ is 
pseudo-effective (this last condition being essentially necessary).

\medskip

\noindent Recall (cf. \cite{CP13}) that the orbifold cotangent bundle is defined by lifting logarithmic differentials with denominators of fractionary exponents to suitable ramified covers
$\pi: X_\Delta\to X$ adapted to the pair 
$(X, \Delta)$. The ramified cover $\pi$ is Galois, and we denote by $G$ the corresponding group.

\begin{theorem}\label{main}
Let $(X, \Delta)$ be a smooth projective log-canonical pair, with pseudo-effective canonical bundle $K_X+ \Delta$. 
Let $\cQ$ be any quotient of the tensor power $\otimes ^m\pi^\star\Omega^1{(X, \Delta)}$, $m\geq 1$ being any integer. 

For any movable class $\alpha$ on $X$, we then have, on $X_{\Delta}$:
\begin{equation}\label{007}
c_1(\cQ).\pi^\star\alpha\geq 0
\end{equation}

\end{theorem}

\noindent If $\Delta=0$, and if $K_X$ is pseudo-effective, this says that the determinant of any quotient of $\otimes ^m\Omega^1_X$ is pseudo-effective, strengthening a fundamental result of Y. Miyaoka (\cite{Mi}) stating that $\Omega^1_X|_C$
is \emph{nef} for any sufficiently generic complete intersection curve $C\subset X$. In \cite{CPe}, Theorem \ref{main} is stated\footnote{As pointed out by A. Langer, there is a very serious gap in the proof of Theorem 1.4 of \cite{CPe}. On page 49, the reference [18] is indeed used in a context which is not covered by [18].  
This does not affects the results of sections 3 and 5 of \cite{CPe}. In fact, all statements of \cite{CPe} are true, as special cases of the ones in the present text.} as Theorem 1.4 when $\Delta=0$.
In \cite{CP13} we obtained the analog of Miyaoka's theorem for log-canonical
orbifolds.

\noindent In order to illustrate the main ideas, we now sketch the proof of Theorem \ref{main} in case $\Delta=0$. 

Assume by contradiction the existence of a sheaf $\cQ$ and a  movable class $\alpha$ as above, such that the inequality \eqref{007} is not satisfied. By dualising, this means that the maximal destabilizing subsheaf $\cF$ of $T_X$ has a positive $\alpha$-slope. The algebraicity criterion of Theorem \ref{rcintro} shows that $\cF$ defines a rational map $p:X\dasharrow Z$, the generic fibre of which is the Zariski-closure of a leaf of $\cF$, and rationally connected. This contradicts  the pseudo-effectivity of $K_X$.

\vskip 7pt

\noindent If $\Delta\neq 0$, then the proof of Theorem \ref{main} requires several constructions of foundational nature. They are related to the notion of holomorphic orbifold tensors, which is exposed in detail in \S 5. As we have already mentioned, the holomorphic tensors corresponding to $(X, \Delta)$ are defined on suitable covers $\pi:X_{\Delta}\to X$.
An equally important technical tool
is the orbifold version of
Lie bracket. We show that the orbifold tangent bundle
is closed with respect to this operation, and we derive an orbifold version of Frobenius integrability criteria. The inverse image
$\pi^\star T_X$ is not a sub-sheaf of $\displaystyle T_{X_\Delta}$, and additional arguments are needed in this broader context.

The difference in the conclusions of Theorem \ref{rco} and Theorem \ref{rcintro} is due to the absence of a theory of rationally connected  objects in
the category of orbifold pairs. 

\

\noindent A consequence of the results we develop in orbifold setting is the
following version of Theorem \ref{rcintro}, as follows.

\begin{theorem}\label{rco} Let $(X, \Delta)$ be a smooth projective log-canonical pair. Let $\cF\subset \pi^*T(X,\Delta)$ be a saturated $G$-invariant subsheaf such that:
\begin{enumerate}

\item[(1)] $\mu_{\alpha,min}(\cF)>0$,
  \smallskip
  
\item[(2)] $\displaystyle \mu_{\alpha,min}(\cF)> \frac{1}{2}
  \mu_{\alpha,max}(\pi^*T(X,\Delta)/\cF)$.

\end{enumerate}
The saturation of $\cF$ in $\pi^*(TX)$ then is equal to the $\pi$-inverse image of a coherent sheaf $\cF'\subset T_X$. 
Moreover, $\cF'$ 
defines an algebraic foliation on $X$ such that the restriction of $K_X+\Delta$ to the closure $F'$ of the generic
leaf of $\cF'$ is not pseudo-effective. \end{theorem}
\medskip

\noindent \emph{Structure of the text.} 

Section $2$ recalls the notions and results needed here about the \emph{stability with respect to a movable class,} introduced in \cite{CPe}. 

Section $3$ studies the positivity properties of the \emph{relative canonical bundle} of a rational map. In particular, its degree on lifts of movable classes is preserved under modifications.This permits a reduction to `neat' models of arbitrary rational fibrations.

Section $4$ establishes Theorems \ref{rcintro} and \ref{capet}.\vskip 4pt

In the next two sections we treat the orbifold version of these results.

Section $5$ reviews the definition of the orbifold (co)tangent bundles. For the smooth log-canonical pairs $(X, \Delta)$ considered here, these objects admit an explicit simple description on suitable ramified covers introduced by Y. Kawamata. 
 
\noindent The notion of \emph{Lie derivative} in orbifold setting is introduced here.
This operator is deduced from the lift of the Lie derivative of $T_X$.
We establish a version of the classical Frobenius integrability criteria, in the following sense. If $\cF_\Delta\subset \pi^\star T(X, \Delta)$ is a
saturated and $G$-invariant subsheaf for which the orbifold Lie bracket vanishes, then $\cF_\Delta$ is the $\pi$-inverse image of a holomorphic foliation
$\cF$ on $X$.  

Section $6$ gives the proofs of Theorems \ref{rco} and \ref{main}, by combining the previous preparatory results. 

Section $7$ deals with the \emph{birational stability} of the orbifold cotangent bundle of
$(X, \Delta)$ if $K_X+\Delta$ is pseudo-effective. This means that the numerical dimension of any sub-line bundle $L$ of $\otimes^m\pi^*\Omega^1(X,\Delta)$ is bounded by the numerical dimension of $K_X+\Delta$. 

Combined in Section 8 with the work of Viehweg-Zuo (\cite{VZ}), these results permit to compare the variation of families of projective manifolds with ample canonical bundles to the canonical bundle on the base of the family. Related results by B. Taji and Popa-Schnell are mentioned (\cite{Taj} and \cite{PoS}).

\

\noindent \emph{We thank B. Claudon, S. Druel, J.V. Pereira, E. Rousseau, B. Taji and M. Toma for comments, advices, corrections and complements on the first version of this text. We are equally grateful to the anonymous referees for important suggestions and constructive criticism which improved substantially the exposition and the mathematical content of this article.}

%%%%%%%%%%%%%%%%%%%%%%%%%%%%%%%%%%%%%%%%%%%%%%%%%%%%%%%%%%%%%%%%%%%%%%%%%%%%%%%%%%%%%%%%%%%%%%%%%%%%%%%%%%%%%%%%%%%%%%%%%%%%%%%%%%%%%%%%%%%%%%%%%%%%%%%%%%%%%%%%%%%%%%%%%%%%%%%%%%%%%%%%%%%%%%%%%%%%%%%%%%%%%%%%%%%%%%%%

\bigskip

\section{Slope and semi-stability with respect to movable classes}
\medskip

\noindent We will collect in this section a few results concerning the notion of 
slope stability of a sheaf with respect to a movable class. They were introduced in \cite{CPe}. 
These results play a crucial role in the 
proof of our algebraicity criteria, cf. section four. See \cite{GKP} for a detailed and extended presentation.\medskip

\subsection{The movable cone}

\noindent To start with, let $N_1(X)_{\bR}$ be the space of numerical curves classes on $X$. We recall the following notion.
\begin{defn}\label{mov}
A class $\alpha\in N_1(X)_{\bR}$ is called movable if we have $\alpha\cdot D\geq 0$ for any effective divisor 
$D$. The set of such classes form a closed convex cone denoted by $\Mov(X)$ and called the movable cone.

A movable class is said to be rational if it belongs to $N_1(X)_{\bQ}$
\end{defn}

By the main result in \cite{BDPP}, the set $\Mov(X)$ is the closed convex cone in $N_1(X)_{\bR}$
generated by the classes $[C]$ of `movable curves', where an irreducible curve is said to be `movable' if it is a member of a covering algebraic family of curves on $X$ parametrised by an irreducible projective variety. This is also the closed cone generated by the classes of curves of the form $\pi_\star(H_1\cap\dots \cap H_{n-1})$, where $\pi: \wh X\to X$ is a modification of the manifold $X$, and the $H_j's$ are hyperplane sections of $\wh X$.
\vskip 0.5cm

%%%%%%%%%%%%%%%%%%%%%%%%%%%%%%%%%%%%%%%%%%%%%%%%%%%%%%%%%%%%%%%%%%%%%%%%%%%%%%%%%%%%%%%%%%%%%%%%%%%%%%%%%%%%%%%%%%%%%%%%%%%%%%%%%%%%%%%%%%%%%%%%%%%%%%%%%%%%%%%%%%%%%%%%%%%%%%%%%%%%%%%%%%%%%%%%%%%%%%%%%%%%%%%%%%%%%%%%

\subsection{Slopes associated to a movable class}

Let $\cE\neq 0$ be a coherent, torsion-free sheaf on $X$; let $\det\cE$ its determinant, that is, the bi-dual of its top power. This is a line bundle on $X$ with first Chern class $c_1(\cE)$.
If $\alpha\in \Mov(X)$ is a movable class the $\alpha$-slope $\mu_{\alpha}(\cE)$ of $\cE$ is:
\begin{equation}\label{01061}
\mu_{\alpha}(\cE):= \frac{c_1(\cE).\alpha}{\rk(\cE)}
\end{equation}
 
\medskip

\noindent The $\alpha$-semi-stability is defined as usual.
\begin{defn}
The torsion-free coherent sheaf $\cE$ is $\alpha$-semistable if
\begin{equation}\label{617}
\mu_{\alpha}(\cG)\leq \mu_{\alpha}(\cE)
\end{equation}
for any non-trivial coherent subsheaf $\cG\subset \cE$.
\end{defn}

\noindent The $\alpha$-stability (not used here) is defined in a similar manner, the 
inequality \eqref{617} being strict if the rank of $\cG$ is strictly smaller than the rank of $\cE$.

As showed in \cite{CPe}, essentially all of the properties of the classical slope-stability theory 
still hold in this extended setting. A crucial exception is the Mehta-Ramanathan theorem (see example 3.9 below). 

The construction of Harder-Narasimhan
filtrations with respect to  movable classes also holds, with the same properties. We only state the results for smooth projective manifolds since this is  the only case needed here. As observed in \cite{GKP}, the theory adapts immediately when the variety $X$ is $\bQ$-factorial. 

\begin{defn}
Let $X$ be a non-singular manifold, and let $\cE$ be a coherent, torsion-free sheaf of positive rank on $X$.
We define:
\begin{equation}\label{618}
\mu_{\alpha, {\rm max}}(\cE):= \sup \{ \mu_\alpha(\cF) : \cF\subset \cE,  \hbox{ any nonzero coherent subsheaf }\}
\end{equation}
as well as its dual version:
\begin{equation}\label{6619}
\mu_{\alpha, {\rm min}}(\cE):= \inf \{ \mu_\alpha(\cQ) : \cE\to \cQ\to 0\}
\end{equation}
where the quotient sheaf $\cQ$ in \eqref{6619} is coherent, non-zero and torsion-free.  
\end{defn}

\medskip

\noindent We quote next the following result.

\begin{proposition}\label{maxdest} \cite{CPe} There exists a non-zero, coherent sheaf, unique and maximal for the inclusion $\cF\subset \cE$ such that we have
\begin{equation}\label{619}
\mu_\alpha(\cF)= \mu_{\alpha, {\rm max}}(\cE).
\end{equation}
The supremum in \eqref{618} is thus a maximum.
\end{proposition}
\medskip

\noindent The sheaf $\cF$ in Proposition \ref{maxdest} is obviously
$\alpha$-semistable; it is is called the maximal destabilizing subsheaf. 
\medskip

\noindent The following simple vanishing criterion for sections of coherent sheaves in terms of the slope function will be used here.

\begin{lemma}\label{negative} {\rm (\cite{CPe})}
Let $\cE$ be a coherent, torsion-free sheaf. 

If $\mu_{\alpha, {\rm max}}(\cE)< 0$ for some movable class $\alpha$, then $H^0(X, \cE)= 0$.

More generally: $\Hom(\cE, \cE')=0$, if $\mu_{\alpha, {\rm min}}(\cE)> \mu_{\alpha, {\rm max}}(\cE').$
\end{lemma}
\noindent For example, the first claim of this lemma applies if $\cE$ is $\alpha$-semistable and of negative slope.

\medskip

\noindent We will use the following in the proof of Theorem \ref{van1}:

\begin{proposition}\label{pi^*} Let $\pi:X'\to X$ a finite Galois ramified cover of group $G$ between complex and connected projective manifolds. Let $\alpha$ be a movable class on $X$, and $\cE$ a vector bundle on $X$, with $\cE':=\pi^*(\cE)$ and $\alpha':=\pi^*(\alpha)$. Then $\mu_{\alpha,max}(\cE)=\mu_{\alpha',max}(\cE')$, and $\cF':=\pi^*(\cF)$ is the maximal destabilizing subsheaf of $\cE'$, if $\cF$ is the maximal destabilizing subsheaf of $\cE$. 
\end{proposition}

\begin{proof} Let $\cF'$ be the maximal destabilizing subsheaf of $\cE'$ (with respect to $\alpha'$). It is $G$-invariant, since so is $\cE'$. Then $\pi_*(\cF')^G\subset \pi_*(\cE')^G=\cE$ is such that $\pi^*(\cF)=\cF'$ and has thus the same $\alpha$-slope as $\cF'$. Since obviously, $\mu_{\alpha,max}(\cE)\leq\mu_{\alpha',max}(\cE')$, the claim is proved.
\end{proof}

%%%%%%%%%%%%%%%%%%%%%%%%%%%%%%%%%%%%%%%%%%%%%%%%%%%%%%%%%%%%%%%%%%%%%%%%%%%%%%%%%%%%%%%%%%%%%%%%%%%%%%%%%%%%%%%%%%%%%%%%%%%%%%%%%%%%%%%%%%%%%%%%%%%%%%%%%%%%%%%%%%%%%%%%%%%%%%%%%%%%%%%%%%%%%%%%%%%%%%%%%%%%%%%%%%%%%%%%

\subsection{Tensor products}

\noindent  If $\cE_1, \cE_2$ be two coherent, torsion-free sheaves on $X$. We denote by $\cE_1\wh \otimes \cE_2$
the reflexive hull $\left(\cE_1\otimes \cE_2\right)^{\star\star}$. 

The following fundamental result was established in \cite{ CPe} if the class $\alpha$ is either rational or in the interior of $\Mov(X)$. When $\cE_j$ are vector bundles, the proof given by Matei Toma relies on the deep analytic Kobayashi-Hitchin correspondance for Gauduchon metrics established by Li-Yau (\cite{LY}). For $\alpha$ arbitrary, the proof given in \cite{GKP}, is treated by reduction to this basic case.

\begin{theorem}\label{tensor} {\rm (\cite{CPe}, \cite{GKP})}
Let $\alpha$ be a movable class on $X$; if $\cE_j$ as above are both $\alpha$-semistable, then so is $\cE_1\wh \otimes \cE_2$.
\end{theorem}
\medskip

\noindent The slope behaves well under tensor operations. We only mention next the few properties used here, and we refer to the articles quoted above for a complete proof.

\begin{proposition}Let $\cE,\cG$ be two torsion-free coherent sheaves, and let $\alpha$ be 
a movable class. Then we have the following properties:

\begin{enumerate}

\item[(1)] The slope of the tensor product equals 
$$\mu_{\alpha}\big(\cE\wh \otimes \cG\big)=\mu_{\alpha}(\cE)+\mu_{\alpha}(\cG).$$

\item[(2)] For each $m\geq 1$ we have
$$\mu_{\alpha}\big(\Sym^m(\cE)\big)^{\star\star}= m\mu_{\alpha}(\cE).$$

\item[(3)] The slope of the exterior product equals
$$\mu_{\alpha}\big(\wedge^2(\cE)\big)^{\star\star}=2 \mu_{\alpha}(\cE).$$

\item[(4)] Moreover, if $\cE$ and $\cG$ are semistable with respect to $\alpha$, then the sheaves $\cE\wh \otimes \cG,
  \big(\Sym^m(\cE)\big)^{\star\star}$ and
  $\left(\wedge^2(\cE)\right)^{\star\star}$
  are equally $\alpha$-semistable.
\end{enumerate}
\end{proposition}
\medskip

\noindent The following statement is established in \cite{GKP}
as consequence of the existence of the
Harder-Narasimhan filtration with respect to a mobile class.

\begin{theorem}\label{tensor, I} {\rm (\cite{CPe}, \cite{GKP})}
Let $\cE,\cG$ be two torsion-free coherent sheaves, and let $\alpha$ be 
a movable class. Then we have
\begin{equation}\label{max/min}
  \mu_{\alpha, max}(\cE\wh \otimes \cG)= \mu_{\alpha, max}(\cE)+
  \mu_{\alpha, max}(\cG), 
\end{equation}
together with the corresponding relations for $\mu_{\alpha, min}$.
Similar identities hold true if we replace the reflexive
tensor product in \eqref{max/min} with $\big(\Sym^m \cE\big)^{\star\star}$ or with $\left(\wedge^2\cE\right)^{\star\star}$.
\end{theorem}

\noindent The results collected here are relevant in the context of the vanishing criterion we discuss next.

\subsection{A vanishing criterion: from exterior to tensor powers}

We consider the following situation: $\pi:X'\to X$ is a finite ramified Galois cover of group $G$ between two connected complex projective manifolds. Let $E'$ be a $G$-invariant holomorphic vector bundle on $X'$, and $L'$ be any numerically trivial line bundle on $X'$. We finally consider also ample movable classes $\alpha=H^{n-1}$, for $H$ varying in an non-empty open subset of the polarisation classes on $X$, with $\alpha':=\pi^*(\alpha)$ their inverse images on $X'$. By \cite{GT}, Proposition 6.5, these $\alpha's$ thus cover a nonempty open subset $U$ in the cone of movable classes on $X$. The proof of \cite{GT} consists in deriving the map $p:H\to H^{n-1}$, the Hard Lefschetz theorem implying that it is submersive at any point $H$\footnote{The authors also show the injectivity of $p$ by a ingenious use of Khovanskii-Teissier inequalities.}.

\begin{theorem}\label{van1} We assume that the following holds

\begin{enumerate}

\item[(1)] $\mu_{\pi^*(\alpha), max}(E')\leq 0$, for any $\alpha\in U$.
\smallskip

\item[(2)]$H^0(X', \wedge^qE'\otimes L')=0$, for any $q>0$, and $L'\equiv 0$ on $X'$. 
\end{enumerate}
Then we have 
$$H^0(X', \otimes^m E'\otimes L')=0,$$ 
for any $m>0$ and $L'\equiv 0$ on $X'$.
\end{theorem}

\noindent Before proceeding to the proof, we remark that if $E'=\pi^*(E)$ for some vector bundle $E$ on $X$, and if $H_1(X,\bZ)=0$\footnote{In particular, if the algebraic fundamental group $\widehat{\pi}_1(X)$ of $X$ is trivial.}, then the hypothesis (2) above can be replaced by the weaker hypothesis:
\begin{enumerate}

\item[(2')]  $H^0(X, \wedge^qE)=0$, for any $q>0$.
\end{enumerate}
and obtain the same conclusion.

\begin{proof} (of Theorem \ref{van1}): Assume by contradiction that we have a non-zero section of $\otimes^m E'\otimes L'$ for some $L'\equiv 0$. Then $\mu_{\alpha',max}(\otimes^m E'\otimes L')=m. \mu_{\alpha',max}(E')=0$, for every $\alpha\in U$. There thus exists an $\alpha$ for which the maximal $\alpha'$-destabilizing subsheaf $\cF'\subset E'$ has maximum rank $q>0$. Because $\mu_{\beta',max}(E')=0$ for $\beta'=\pi^*(\beta)$ with $\beta$ close to $\alpha$, we see\footnote{This clever observation was communicated to us by Matei Toma.} that $\cF'$ is also the $\beta'$-maximal  destabilizing subsheaf of $E'$ (just write $\beta=\alpha+t.\gamma$, $t>0$, and replace $t$ by -$t$, just as in the proof of the vanishing of a derivative in a local maximum). Since $\cF'$ is $G$-invariant, so is $det(\cF')$, and so $N.det(\cF')=\pi^*(L)$, for some $L\in Pic(X)$, if $N:=Card(G)=deg(\pi)$. Thus $N.det(\cF').\pi^*(\alpha)=\pi^*(L).\pi^*(\alpha)=N.L.\alpha=0, \forall \alpha\in U$, and: $L\equiv 0$, so that, also: $det(\cF')\equiv 0$. Since $det(\cF')\subset \wedge^qE'$, we get that $H^0(E',\wedge^qE'\otimes L')\neq 0$, if $L'=-\det(\cF')$. This contradicts the hypothesis 2, and proves the theorem. 

Let us show how to modify the proof in order to get the conclusion from the hypothesis 2' if $E'=\pi^*(E)$ and if $H_1(X,\bZ)=0$: in this case indeed, by Proposition \ref{pi^*}, $\mu_{\pi^*(\alpha),max}(\pi^*(E))=\mu_{\alpha,max}(E)$ and the maximal $\pi^*(\alpha)$-destabilizing subsheaf $\cF'$ of $\pi^*(E)$ is the inverse image by $\pi$ of the maximal destabilizing subsheaf $\cF$ of $E$. We thus obtain $\cF\subset E$ such that $det(\cF)\equiv0$, so that $\cO_X\cong \det(\cF)$, because $H_1(X,\bZ)=0$, and so: $\cO_X\cong \det(\cF)\subset \wedge^q(E)$, contradicting 2'. \end{proof}
\medskip

\noindent An illustration of the applications of this
result is the following statement.

\begin{corollary}\label{corvan1} \cite{Camp16} Let $(X,D)$ be a smooth orbifold pair with $X$ projective smooth and $D$ a reduced divisor on $X$ with simple normal crossings. Assume that:
\begin{enumerate}

\item[(1)] $(X,D)$ is Fano (ie: $-(K_X+D)$ is ample);
\smallskip

\item[(2)] $H^0(X,\Omega^q_X(Log(D))\otimes L)=0$, for $q>0$ and any $L\equiv 0$ in $Pic(X)$. 
\end{enumerate}
Then we have
$H^0(X, \otimes^m \Omega^1_X(Log(D))\otimes L)=0$
for any $m>0$ and $L\equiv 0$ on $X$.
\end{corollary}

\noindent For the proof we refer to \cite{Camp16}.
\medskip

\noindent Remark that, in general, $H^0(X,\Omega^q_X(Log(D)))\neq 0$ for Fano pairs $(X,D)$ as above (as shown by $(\Bbb P^n,D)$, if $D$ is a union of $k\leq n$ hyperplanes in general position, for which $h^0(X,\Omega^1_X(Log(D)))=k.)$

\begin{question} Let $(X,D)$ be Fano as above. Does the conclusion of corollary \ref{corvan1} hold if one only assumes that $H^0(X,\Omega^1_X(Log(D)))=0$\footnote{Instead of: $H^0(X,\Omega^q_X(Log(D))\otimes L))=0$, for $q>0$ and any $L\equiv 0$ in $Pic(X)$.}? 
One may indeed wonder whether the Quasi-Albanese map is the only obstruction to the vanishing of $H^0(X, \otimes^m \Omega^1_X(Log(D)))$ for any $m>0$.
\end{question}

%%%%%%%%%%%%%%%%%%%%%%%%%%%%%%%%%%%%%%%%%%%%%%%%%%%%%%%%%%%%%%%%%%%%%%%%%%%%%%%%%%%%%%%%%%%%%%%%%%%%%%%%%%%%%%%%%%%%%%%%%%%%%%%%%%%%%%%%%%%%%%%%%%%%%%%%%%%%%%%%%%%%%%%%%%%%%%%%%%%%%%%%%%%%%%%%%%%%%%%%%%%%%%%%%%%%%%%%

\subsection{Birational invariance of slope-positive foliations.}

\noindent We consider a saturated distribution $\cF\subset TX$, and a birational morphism $\pi:\wh X\to X$,
where $\wh X$ is also non-singular. Then we get an induced distribution on $\wh \cF\subset T\wh X$, as follows. The tangent bundle of $\wh X$ can be seen as subsheaf of the $\pi$-inverse image
$\pi^\star(T_X)$ of the tangent bundle of $X$, and we define
$\wh \cF:= \pi^\star(\cF)\cap T_{\wh X}$. 
\smallskip

\noindent We establish next the 
preservation of the slopes under 
birational modifications. Although very simple, this observation is fundamental. It is also noticed in the very recent article \cite{Dr}. It was already stated in \cite{CPe}, Section 5, but not used in the context of foliations.

\begin{lemma}\label{bmq} Let $\pi:\wh X\to X$ be a birational morphism between two smooth and connected complex projective manifolds.

Let $\cF\subset TX$ be a saturated distribution, and $\wh \cF:=\pi^*(\cF)\cap T\wh X$ be its inverse 
image in $T\wh X$. Let $\alpha$ be a movable class on $X$, and $\wh \alpha:= \pi^\star \alpha$ be its inverse image on $\wh X$. Then $\wh \alpha\in \Mov(\wh X)$ is a movable class on $\wh X$. Moreover, we have 
\begin{equation}\label{620}
\mu_{\alpha}(\cF)=\mu_{\wh\alpha}(\wh \cF)
\end{equation} 
\end{lemma} 
\noindent 

\begin{proof} The fact that $\wh \alpha$ is a movable class on $\wh X$ is a direct consequence of \cite{BDPP}.
We have $\mu_{\alpha}(\cF)=\mu_{\wh \alpha}(\pi^\star\cF)$; on the other hand, $\det(\wh\cF)$ and $\det(\pi^\star\cF')$ differ by an (effective) $\pi$-exceptional divisor $E$ on $\wh X$. Since 
$\wh \alpha \cdot E= 0$ for any such divisor, the statement is proved. 
\end{proof}

\begin{rem}{\rm
In particular, both slopes in \eqref{620} are simultaneously positive, negative, or zero \emph{provided that the class
$\wh\alpha$ is the inverse image of the movable class $\alpha$ on $X$}. 
However, this type of preservation of slope-positivity with respect to movable classes $\beta$ and 
$\pi_\star\beta$ on $\wh X$ and $X$ respectively might fail, 
as illustrated by the 
following example. Let $\cF$ be the foliation given by a generic pencil of conics on $\bP^2$. The slope is negative, but becomes positive on the blow-up $\wh \bP^2$ of the four base-points, if one chooses for $\beta$ on $\wh \bP^2$ any ample class. We thank J.\ Pereira for this observation and for this example.}
\end{rem} 
\medskip

\begin{rem}\label{sat}
  {\rm Let $\cF_1\subset \cF_2$ be two torsion-free coherent sheaves having the same rank (we recall that in this context,
the sheaves $\cF_i$ are locally free outside a subset of co-dimension at least two, and their respective rank is defined via the associated vector bundles). If $\mu_{\alpha, \rm min}(\cF_1)> 0$, then we equally have $\mu_{\alpha, \rm min}(\cF_2)> 0$.
This is a consequence of the fact that $\det(\cF_2)= \det(\cF_1)\otimes \cO(D)$ for some 
\emph{effective} divisor $D$. Notice, by contrast, that if $\cF_1$ is semi-stable, $\cF_2$ need not be semi-stable (see Remark \ref{rstab}).
}
\end{rem}

\medskip

%%%%%%%%%%%%%%%%%%%%%%%%%%%%%%%%%%%%%%%%%%%%%%%%%%%%%%%%%%%%%%%%%%%%%%%%%%%%%%%%%%%%%%%%%%%%%%%%%%%%%%%%%%%%%%%%%%%%%%%%%%%%%%%%%%%%%%%%%%%%%%%%%%%%%%%%%%%%%%%%%%%%%%%%%%%%%%%%%%%%%%%%%%%%%%%%%%%%%%%%%%%%%%%%%%%%%%%%

\section{Pseudoeffectivity of relative canonical bundles}
\smallskip

\noindent Let $p: X\dasharrow Z$ be a dominant and connected rational map, where $X$ and $Z$ are non-singular projective 
manifolds. In this section we define the `saturated' relative canonical bundle of $p$, and establish some of its birational positivity properties with respect to a movable class. This relies in particular on the preceding observation that changing both $X$ and the birational model $Z$, and lifting $\alpha$, the slopes of the corresponding sheaves are preserved. 

Let $X_0\subset X$ be the largest Zariski open set such that the 
restriction 
\begin{equation}\label{623}
p|_{X_0}: X_0\to Z
\end{equation}
of our given rational map $p$ is holomorphic; in particular we have $\codim_{X}X_0\geq~2$. The map \eqref{623} above will be denoted by $p_0$ in the sequel.

\begin{defn}\label{carat}
Let $K_Z$ be any divisor on $Z$ in the canonical class. Let $p^\star K_Z$ be the 
closure $\overline{p_0^\star K_Z}$ of the analytic cycle $p_0^\star K_Z$ of $X_0$. The relative canonical bundle of $p$ is 
\begin{equation}\label{624}
K_{X/Z}:= K_X- p^\star K_Z.
\end{equation}
\end{defn}

\noindent We introduce next the divisor $D(p)$ on $X$ by:
\begin{equation}\label{625}
D(p):=\sum_k(t_k-1)F_k,
\end{equation}
where the hypersurfaces $F_k$ in \eqref{625} are all the irreducible divisors of $X$ which, restricted to $X_0$, are mapped by $p_0$ to divisors $G_k\subset Z$, such that $p_0^\star G_k$ vanishes to the order $t_k\geq 2$ along $F_k$. It is a key object in the study of \emph{holomorphic foliations}, whose definition will be recalled next.

\begin{defn}\label{fol} A foliation on a manifold $X$ is a coherent subsheaf $\cF\subset T_X$ enjoying the following
properties:

\begin{enumerate}

\item[(i)] $\cF$ is closed under the Lie bracket, and
\smallskip

\item[(ii)] The quotient $T_X/\cF$ is torsion-free, i.e. $\cF$ is saturated in $T_X$.
\end{enumerate}
\end{defn}

\noindent Let $X_0\subset X$ be the maximal open subset of $X$ such that the restriction 
$\displaystyle \cF|_{X_0}$ is a sub-bundle. We note that the codimension of the complement $X\setminus X_0$ in 
$X$ is at least two, given that $\cF$ is torsion-free. A \emph{leaf} of $\cF$ is a connected, locally closed holomorphic sub-manifold $L\subset X_0$ whose tangent bundle coincides with $\cF$, i.e. $T_L= \cF|_L$.
A leaf $L$ is called algebraic if it is open in its Zariski closure.
\smallskip

\noindent For example, the kernel of the differential of a rational map $p: X\dasharrow Z$ defines a foliation on $X$, whose leaves are algebraic.
Even if $\Delta= 0$, the relevance of the divisor $D(p)$ to the study of foliations is explained by the following remark \ref{KF}. This is certainly well-known to experts. We will not give the proof of this statement here, because 
the more general orbifold version will be established in
\ref{claim2}.

\begin{rem}\label{KF} {\rm Let $p:X\dasharrow Z$ be a dominant rational fibration, and let $\cF$ be a foliation on $X$  such that $\cF= \Ker(dp)$. Let $\pi_X: \wh X\to X$ and $\pi_Z: \wh Z\to Z$ be modifications of 
$X$ and $Z$, respectively, such that the following properties are satisfied:
\begin{enumerate}

\item[(i)] The induced map $\wh p: \wh X\to \wh Z$ is regular, its discriminant locus $ E$ is a snc divisor and so it is the inverse image $\wh p^{-1}(E)$.

\smallskip

\item[(ii)] If a component $W$ of $\wh p^{-1}(E)$ is $\wh p$-exceptional, then it is also $\pi_X$-exceptional.
\end{enumerate}
Let $\wh \cF$ be the foliation induced by $\cF$ on $\wh X$; then we have the equality 

\begin{equation}\label{722}
K_{\wh \cF}= K_{\wh X/\wh Z}- D(\wh p).
\end{equation}
modulo a divisor which is $\pi_X$- exceptional,
cf. Lemma \ref{claim2} and \cite{Dr}. In particular, if $K_X$ is pseudo-effective, then so it is 
$K_{\cF}$, by the crucial theorem \ref{pseffff} below.
}
\end{rem}

\noindent In the previous remark, we denote by 
\begin{equation}\label{1001}
K_{\cF}:= \det(\cF^\star)
\end{equation}
the \emph{canonical bundle of a foliation} $\cF$; the determinant above is the bi-dual of the maximum exterior power of $\cF^\star$.  

\medskip

\noindent The main result of this section is the following one. A similar observation is made in \cite{Dr}, Proposition 4.3 and the references there.

\begin{theorem}\label{pseffff} Let $(X, \Delta)$ be a lc pair, such that $X$ is smooth and such that 
$\Delta$ is snc. Assume that $K_X+ \Delta$ is pseudo-effective.
Then for any rational map $p$ as in Remark \ref{KF}, the divisor 
$$K_{X/Z}+ \Delta^{\rm hor }-D(p)$$ is pseudo-effective.
\end{theorem}

\noindent In the statement \ref{pseffff} above we denote by $\Delta^{\rm hor }$ the the divisor having the same multiplicities as 
$\Delta$ on the irreducible hypersurfaces of $X$ which project onto $Y$ via the map $p$, and zero for any other hypersurfaces.

\begin{proof} We shall deduce this statement from Theorem 2.11 in \cite{CP13}. 

Consider a holomorphic birational model of $p$. There exists a modifications $\pi_X: \wh X\to X$ of $X$ 
for which the next properties hold:

\begin{enumerate}

\item[(1)] The induced map $\wh p: \wh X\to Z$ is holomorphic.

\item[(2)] The $\pi_X$ inverse image of $\Delta$ is snc. 

\end{enumerate}

Define the divisor $\wh \Delta$ by the usual formula:
\begin{equation}\label{626}
E_1+ \pi_X^\star(K_X+ \Delta)= K_{\wh X}+ \wh \Delta
\end{equation}
where $E_1$ is effective and $\pi_X$-exceptional and $(\wh X, \wh \Delta)$ is lc.

By definition \ref{carat} we deduce that we have the equality
\begin{equation}\label{627}
\wh p^{\star}K_Z= \pi_X^\star\big(p^\star K_Z\big)+ E_2
\end{equation}
where $E_2$ is a $\pi_X$- exceptional divisor.

Combining \eqref{626} and \eqref{627} we get:
\begin{equation}\label{628}
E_1+ \pi_X^\star\big(K_{X/Z}+ \Delta\big)= K_{\wh X/Z}+ \wh \Delta+ E_2,
\end{equation}  
which is preserved when taking into account the multiplicity divisors of the maps
$p$ and $\wh p$:
\begin{equation}\label{629}
E_1+ \pi_X^\star\big(K_{X/Z}+ \Delta- D(p)\big)= K_{\wh X/Z}+ \wh \Delta- D(\wh p)+ E_2.
\end{equation}
Notice however that the divisors $(E_j)$ in \eqref{628} and \eqref{629} may be different, but for the notation simplicity we keep the same symbols. The point is that both of them are $\pi_X$- exceptional and effective.

Next we use the pseudo-effectivity theorem in \cite{CP13}, which implies that the $\bQ$-line bundle
\begin{equation}\label{630}
K_{\wh X/Z}+ \wh \Delta- D(\wh p)
\end{equation}
is  pseudo-effective on $\wh X$ (we remark that at this point the hypothesis $(X, \Delta)$ \emph{is log-canonical} is used in an essential manner). The hypothesis in the statement \cite{CP13} are indeed satisfied, since for any $z\in Z$ generic 
the restriction $K_{\wh X_z}+ \wh \Delta|_{X_z}$ is pseudo-effective, since so is $K_X+\Delta$, and thus also $K_{\wh X}+ \wh \Delta$.

\noindent The conclusion follows from the following simple statement.

\begin {lemma}\label{l pseffff} Let $\pi:\wh X\to X$ be a modification between projective manifolds. Let $\wh L,
L$ be line bundles on $\wh X$ and $X$ respectively. Assume that:
\begin{equation}\label{631}
\wh L= \pi^*L+E_1
\end{equation}
for some $\pi$-exceptional divisor $E_1$ on $\wh X$. 
If $\wh L$ is pseudo-effective, then so is $L$.
\end{lemma}

\begin{proof} Let $\gamma$ be a movable class on $X$. Then $\pi^\star\gamma$ is a movable class on 
$\wh X$, and then we have
$$c_1(\wh L).\pi^\star\gamma\geq 0.$$
By relation \eqref{631}, we deduce that
$$c_1(L).\gamma\geq 0$$
since $E_1\cdot \pi^\star\gamma= 0$, $E_1$ being exceptional.

Thus $L$ is pseudo-effective, by \cite{BDPP}.\end{proof}
\smallskip

\noindent The following alternative arguments for Lemma \ref{l pseffff}
were kindly pointed out to us by the referees. We reproduce them here (in arbitrary order), for the benefit of the readers.

\noindent $\bullet$ We have
$\big(\pi_\star (\wh L)\big)^{\star\star}= L$, as it follows immediately from the assumptions of \ref{l pseffff}. Since the push-forward of a pseudo-effective class is pseudo-effective, we are done. This has the advantages of avoiding the use of \cite{BDPP}.

\noindent $\bullet$ We consider an ample line bundle $\wh A:= \pi^\star(A)- E$ on $\wh X$, where $E$ is effective and $\pi$-exceptional, and $A$ is ample on $X$. Then for any couple of positive integers
the $k\gg m$ the bundle $k\wh L+ m\wh A$ has non-identically zero sections. They are induced by the sections of $kL+ mA$ (since $E_1$ is exceptional), and the proof is finished.

\medskip

\noindent The proof of Theorem \ref{pseffff} is therefore finished, by \eqref{629} combined with Lemma \ref{l pseffff}.
\end{proof}

\smallskip

\begin{rem}\label{repref} {\rm
In the statement \ref{pseffff} above, if the dominant rational map $p:X\dasharrow Z$ is given, then the pseudo-effectivity of the bundle 
$\displaystyle K_{X/Z}+ \Delta^{\rm hor }-D(p)$ is in fact equivalent to the pseudo-effectivity of $\displaystyle K_{\wh X_z}+
\Delta|_{\wh X_z}$ for all points $z$ in the complement of a Zariski closed subset of
$Z$,
cf.\cite{CP13} (here we use the notations in the proof of \ref{pseffff}).
From this perspective, the hypothesis ``$K_X+ \Delta$ pseudo-effective" of Theorem \ref{pseffff}
may look abusive. However 
the point is that this hypothesis insures the pseudo-effectivity of  $\displaystyle K_{X/Z}+ \Delta^{\rm hor }-D(p)$
even if the rational map $p$ is not given a-priori (and it will be the case in what follows).}
\end{rem}

%%%%%%%%%%%%%%%%%%%%%%%%%%%%%%%%%%%%%%%%%%%%%%%%%%%%%%%%%%%%%%%%%%%%%%%%%%%%%%%%%%%%%%%%%%%%%%%%%%%%%%%%%%%%%%%%%%%%%%%%%%%%%%%%%%%%%%%%%%%%%%%%%%%%%%%%%%%%%%%%%%%%%%%%%%%%%%%%%%%%%%%%%%%%%%%%%%%%%%%%%%%%%%%%%%%%%%%%

\section{Algebraicity criteria for foliations}
\medskip

\noindent We begin this section by introducing the following notion --which maybe not standard, but it is very convenient for us.

\begin{defn}\label{alg}
Let $\cF\subset TX$ be a holomorphic foliation of rank $r$. We say that the foliation $\cF$ is \emph{algebraic} if it is induced by a rational map
i.e. $\cF= \ker(dp)$ generically on $X$, for
some dominant rational map $p:X\dasharrow Z$.
\end{defn}
\noindent If this is the case, we see that 
all the leaves of $\cF$ are algebraic subsets of $X$.

\medskip

\noindent The main result of this section is the following statement.

\begin{theorem}\label{algebraic}
Let $X$ be a smooth projective manifold, and let $\cF$ be a foliation on $X$ such that 
there exists a movable class $\alpha$ for which we have
\begin{equation}\label{061}
\mu_{\alpha,min}(\cF)>0.\end{equation} 
The following assertions hold true.
\begin{enumerate}

\item[(1)] The foliation $\cF$ is algebraic.
\smallskip

\item[(2)] The closure of every leaf of $\cF$ is rationally connected.
\end{enumerate}

\end{theorem}

If the class $\alpha$ is a complete intersection of ample hypersurfaces on $X$,  and if $\cF$ is $\alpha$-semistable, 
Theorem \ref{algebraic} is due to
Bogomolov-McQuillan cf. \cite{BMQ}, as well as
Kebekus, Sola-Conde and Toma in \cite{KST}. We equally refer to the article by J.-B. Bost, \cite{Bo}, who proves a different, but related, result in an arithmetic context. These results originate in \cite{Ha}, and \cite{BMQ} is motivated by \cite{Mi}. 
As already mentioned, the approach of the proof below for claim (1) is the same as in \cite{BMQ}. The main difference is that we work directly on $X$ and not by restricting $\cF$ to complete intersection curves. In this way, the Mehta-Ramanathan theorem is not needed, and we avoid the inextricable difficulties generated by both the singularities of $\cF$, and the singularities of covering families of movable curves at their base loci. Notice further that the Mehta-Ramanathan theorem fails for movable curves, cf. Example \ref{exmr} below, already mentioned in \cite{CPe}. The extension from ``generic complete intersection class" to ``movable class" enlarges considerably the potential applicability.

\begin{ex}\label{exmr} {\rm The Mehta-Ramanathan restriction theorem may fail to hold quite drastically even for `strongly' covering families of curves on surfaces. It is indeed shown in \cite{BDPP}, \S7, that if $S$ is a smooth $K3$-surface, then $\cO_P(1)$ is not pseudoeffective on $P:=\Bbb P(\Omega^1_S)$. This means that there exists on $S$ an algebraic family of irreducible curves $C_t$ on $S$ effectively parametrised by a quasi-projective irreducible surface $T$ such that, for each such curve $C_t$ the saturation in $TS$ of the tangent sheaf to $C_t$  has positive degree on $C_t$. Moreover, for $x\in S$ generic, all but a finite number of tangent directions of $TS$ at $x$ are realised by the tangent directions to the $C_t's$ going through $x$. The proof given in \cite{BDPP} is quite indirect. It were interesting to have concrete realisations of such families $C_t$ even on special $K3's$.} \end{ex}

\noindent We shall next prove claim (1); the claim (2) will be established in the subsection \ref{rat}.

%%%%%%%%%%%%%%%%%%%%%%%%%%%%%%%%%%%%%%%%%%%%%%%%%%%%%%%%%%%%%%%%%%%%%%%%%%%%%%%%%%%%%%%%%%%%%%%%%%%%%%%%%%%%%%%%%%%%%%%%%%%%%%%%%%%%%%%%%%%%%%%%%%%%%%%%%%%%%%%%%%%%%%%%%%%%%%%%%%%%%%%%%%%%%%%%%%%%%%%%%%%%%%%%%%%%%%%%

\subsection{Algebraicity}

\begin{proof}
 Let $ E\subset X$ be the singular set of the foliation $\cF$. By definition, it consists of
points where $\cF$ is not a subbundle ot $TX$, and in particular: 
\begin{equation}
\codim_{X}( E)\geq 2.
\end{equation}
Let $x\in X\setminus  E$. Since $\cF$ is not singular at $x$,
there exists an open set $\Omega_{x}\subset X\setminus E$ together with
a submersion $\pi_x: \Omega_x\to \bC^{n-r}$ with connected fibers such that for each $y\in \Omega_x$ the intersection $L_y\cap \Omega_x$ of the
leaf of $\cF$ passing
through $y$ with $\Omega_x$ is given by the fiber of $\pi_x$ containing $y$. We recall that $n= \dim (X)$ and $r$ is the rank of the foliation $\cF$.

Thus we have a cover of the open set $X\setminus E$  
with 
open sets $\Omega_x$ as above. Let
$(\Omega_i)_{i\in I}$ be a countable, locally finite cover extracted from
$\left(\Omega_x\right)_{x\in X\setminus E}$. We define
$$\wt \Omega:= \cup_{i\in I}\Omega_i\times \Omega_i\subset
\left(X\setminus E\right)\times \left(X\setminus E\right);$$ it
is an open subset.  

We define the following $n+r$-dimensional locally closed analytic subset 
$\Lambda\subset \wt \Omega$
as follows
\begin{equation}
\Lambda:= \{(z, w)\in X\times X : z\in \Omega_i \hbox{ and } w\in L_z\cap \Omega_i\hbox{ for some } i\in I\}.
\end{equation}
We note that the (local) analyticity of $\Lambda$ is a direct consequence of the
fact that $\cF$ is a holomorphic foliation.

The set $\Lambda$ contains the open subset of the diagonal defined by:

\begin{equation} X_0:= \{(z, z)\in X\times X :  z\in X\setminus  E\} 
\end{equation}

and we consider
\begin{equation}\label{1102}
V:= \ol\Lambda^{\rm Zar}
\end{equation}
the Zariski closure of $\Lambda$ in $X\times X$. 
\medskip

\noindent We have $\dim(V)\geq \dim(\Lambda)=n+r$, and
we show next that the algebraicity of $\cF$ is
equivalent to the equality $\dim(V)=n+r$. Indeed, if this holds true,
then $\Lambda$ is open in its Zariski-closure
$V$ in $X\times X$. We consider the map $\pi_V: V\to X$ given by the restriction to $V$ of the projection on the first factor $X\times X\to X$. Note that the generic fibers of $\pi_V$ are irreducible, of dimension equal to $r$ (they correspond to the Zariski
closure of the leaves of $\cF$). 

Let $\tau: \wh V\to V$ be a desingularisation of $V$, and let $W\subset \wh V$ be the component of $\wh V$ which contains the inverse image of the
generic fibers of $\pi_V$. We denote by $f:W\to X$ the composed map
$\pi_V\circ\tau|_W$; it is
surjective and by general results, there exists a constant $d> 0$ such that the degree of  
each fiber of $f$ is smaller than $d$. 

We consider the Chow scheme $\Chow(W)= \cup_{\delta > 0}\Chow_{r, \delta}(W)$
corresponding to $r$-dimensional cycles of $W$ (where
the index $\delta$ above stands for the degree of the cycle). The rational map
\begin{equation}\label{cw}
p:X\dasharrow \Chow (W) \quad x\to f^{-1}(x)
\end{equation}  
induces the foliation $\cF$ generically, and we are done by the compactness of the components of the Barlet-Chow scheme of $X$.
\medskip

\noindent The algebraicity of $\cF$ will then follow from the next standard Riemann-Roch bound on sections:

\begin{lemma}\label{testalg} Let $X_0\subset \Lambda\subset V\subset X\times X$ be defined as above. If, for some ample line bundle $L$ on $X\times X$, there exists  a constant $C> 0$ such that
$\displaystyle h^0(V, kL|_V)\leq Ck^{n+r}$ as $k\to \infty,$ then the dimension of $V$ is equal to $n+r$.
\end{lemma}

\noindent We show now the existence of such a constant $C_L=C>0$ for any ample $L$.
\begin{proposition}\label{bounded}
Let $L$ be an ample line bundle on $X\times X$. There exists a constant $C>0$ such that:
$h^0(V, kL|_V)\leq Ck^{n+r},$ for any $k\geq 0$. As a consequence, the dimension of the algebraic set $V$ is equal to $n+r$.  
\end{proposition}

\begin{proof} 
The main ideas in the proof of Proposition \ref{bounded} are the same as in \cite{BMQ}: for any $k\geq 0$, the sections of $L^k$ on $V$ restrict injectively to $\Lambda$, since $V$ is its Zariski closure. Next, one considers the restriction of these sections to the formal neighborhood of $X_0$ in $\Lambda$.  
In other words, we study the Taylor expansion of sections of $L^k|_{\Lambda}$ at the points of the diagonal $X_0$ in the normal directions
in $\Lambda$. 

For any $m>0$, let $X_m$ be the $m^{\rm th}$ infinitesimal neighborhood of $X_0$ in $\Lambda$, defined by the structure sheaf:
$\cO_{X_m}:=\cO_{\Lambda}/I_{0}^{m+1},$
where $I_{0}$ is the sheaf of ideals of the diagonal $X_0\subset \Lambda$. 
It is enough to produce a bound $C>0$ independent of $m, k$ such that 
\begin{equation}\label{1104}
h^0(X_m, L^{\otimes k}\otimes \cO_{X_m})\leq Ck^{n+r}
\end{equation}
for any $k, m$. Indeed, the space $H^0(X_m, L^{\otimes k}\otimes \cO_{X_m})$ is nothing, but the space of all possible Taylor expansions at order $m$ of sections of $L^k$ along $X_0$ in the directions of $\cF$. 

\noindent For this, remark that 
over $(X\setminus  E)$, we have a natural isomorphism $\cF\cong N_{X_0/\Lambda}$, since the normal bundle 
of $X_0$ in $\Lambda$ is naturally isomorphic to the vector bundle corresponding to $\cF|_{X\setminus E}$. 

The following exact sequence holds over $X_0$, $\cF^*$ being the dual of $\cF$:

\begin{equation}\label{1105}
0\to \Sym^m(\cF^*)\to \cO_{X_{m+1}}\to \cO_{X_m}\to 0.
\end{equation}

It shows that it is sufficient to establish that there exists a constant $C>0$ such that, for any $k\geq 0$:

\begin{equation}\label{1106}
\sum_{m\geq 0} h^0(X_0, L^k\otimes Sym^m(\cF^*))\leq C.k^{n+r}.
\end{equation}

\noindent The estimate \eqref{1106} will be a consequence
of following statement.

\begin{lemma}\label{bound}  Let $\cF\subset T_X$ be a coherent sheaf, which is locally free when restricted to the open set $X_0\subset X$ such that $\codim_X(X\setminus X_0)\geq 2$.
Let $\displaystyle \delta_0>\frac{L.\alpha}{\mu_{\alpha,min}(\cF)}$ be any positive integer. The following assertions are true.

\begin{enumerate}

\item[(a)] We have $H^0\left(X_0,L^k\otimes Sym^m(\cF^*)\right)=0$ if
  $m\geq \delta_0k$.

%\item[(b)]  There exists a projective modification $\pi:X'\to X$ such that the quotient of $\cE':=\pi^*(TX^*)$  by  the saturation of $\pi^*((TX/\cF)^*)$ is a vector bundle, denoted $\cQ'$. 

\smallskip

\item[(b)]  There exists a non-singular projective manifold $Y$ of dimension
$\dim(Y)= \dim(X)+ \rk(\cF)-1$ together with a map $p: Y\to X$ and a line bundle $B\to Y$ such that
we have   
\begin{equation}\label{anym}
p_\star(B^{m})= \wh{S^m}(\cF^\star)
\end{equation}  
for any $m\geq 1$. In \eqref{anym} we denote by $\wh{S^m}(\cF^\star)$ the double dual of the symmetric power $Sym^m(\cF^\star)$.
\smallskip

\item[(c)] For any pair of positive integers $k, m$ we have the
  equality
\begin{equation}\label{ineq}
  h^0\left(X_0, L^k\otimes Sym^m(\cF^\star)\right)=
  h^0\left(Y, p^\star(L^k)\otimes B^m\right).
\end{equation}    

%\item[(c)] $\forall k\geq 0,\forall m\geq 0$, we have:

%\begin{equation}\label{687}
%h^0\big(X_0, L^k\otimes \Sym^{m}(\cF^*))\leq  h^0\big(X', \pi^*(L^k)\otimes \Sym^{m}(\cQ'))\end{equation}

%\smallskip

%\item[(d)] $\forall k\geq 0,\forall m\geq B.k$, we have, for a suitable constant $C'>0$:

%\begin{equation}\label{688}
%h^0\big(X_0, L^k\otimes \Sym^{m}(\cF^*))\leq  h^0\big(X', \pi^*(L^k)\otimes \Sym^{m}(\cQ'))\leq C'. k^{n+r-1} \end{equation}

\end{enumerate}
\end{lemma}

\noindent Before proving Lemma \ref{bound}, we notice that it
implies almost immediately the
inequality \eqref{1106}. Indeed, we have

\begin{equation}\label{ex1}
  \sum_{m\geq 0} h^0\left(X_0, L^k\otimes Sym^m(\cF^\star)\right)=
  \sum_{m\leq \delta_0k} h^0\left(X_0, L^k\otimes Sym^m(\cF^\star)\right)
\end{equation}  
by the point (a) of \ref{bound}. Next, the point (c), together with the fact that the dimension of $Y$ is equal to $n+r-1$ shows that the right hand side of  
\eqref{ex1} is $\cO(k^{n+r})$. This can be seen as
follows: the dimension of the
space of global sections of the bundle
$p^\star(L^k)\otimes B^m$ is smaller that $h^0\big(Y, p^\star(L^k)\otimes H^m\big)$, where $H$ is a very ample bundle on $Y$ such that $H\otimes B^{-1}$ is effective. We therefore have to
evaluate the quantity
\begin{equation}\label{1ex1}
\sum_{m\leq \delta_0k} h^0\left(Y, p^\star(L^k)\otimes H^m\right)
\end{equation}  
which is smaller than $\delta_0k\cdot h^0\left(Y, p^\star(L^k)\otimes H^{\delta_0k}\right)$ where we recall that $\delta_0$ is a
positive integer.
By Riemann-Roch theorem, we have the estimate
$h^0\left(Y, p^\star(L^k)\otimes H^{Ck}\right)= \cO(k^{n+r-1})$
as $k\to \infty$, so all in all we have established \eqref{1106}.
\medskip

\noindent In what follows we will identify $X$ with the diagonal of $X\times X$, and $X_0$ with $X\setminus E$.

\begin{proof} 

\noindent The point (a) follows from  Lemma 2.5 and the slope inequality, if $m\geq k.B$: 
\begin{equation}\label{1ex12}
\mu_{\alpha,max} \big(L^k\otimes \wh{S^m}(\cF^\star)\big)=
k.L.\alpha-m.\mu_{\alpha,min}(\cF)<0.
\end{equation}
We remark that here we have used Theorem \ref{tensor, I}.
\smallskip

\noindent The point (b) is completely proved in the book by N.~Nakayama
\cite{Nob}
(cf.\ Chapter V, section 3.23), so we will simply recall the construction of
$(Y, B)$ for the convenience of the reader.

Let $\pi:\bP(\cF^\star)\to X$ be the
scheme over $X$ associated to the torsion free coherent sheaf $\cF^\star$,
and let $\cO_{\cF^\star}(1)$ be the tautological line bundle on $\bP(\cF^\star)$.
Let $\bP^\prime(\cF^\star)$ be the normalization of the component of
$\bP(\cF^\star)$ which contains the Zariski open subset $\pi^{-1}(X_0)$ (we recall the crucial fact that the co-dimension of $X_0$ in $X$ is greater than two).
Finally, 
let $Y$ be a smooth projective variety such that there exists a birational morphism $Y\to \bP^\prime(\cF^\star)$ which is biholomorphic over $\pi^{-1}(X_0)$.
We denote by $\mu: Y\to \bP(\cF^\star)$ the resulting map, and let
\begin{equation}\label{ex2}
p:Y\to X
\end{equation}  
be the composition $\pi\circ \mu$. Nakayama shows that we can take
\begin{equation}\label{ex3}
  B:= \mu^\star\left(\cO_{\cF^\star}(1)\right)+ \Lambda,
\end{equation}
where $\Lambda$ is an
effective $p$-exceptional divisor. The important fact here (cf. \cite{Nob})
is that $B$ 
can be chosen so that \eqref{anym} holds \emph{for any} $m$.
\smallskip

\noindent The equality \eqref{ineq} is a direct consequence of (b), together
with the definition of the set $X_0$, so we do not provide any further explanations.

%(b) Let $\varphi: X\dasharrow Grass_{n-r}(TX)$ be the rational map sending any $x\in X_0$ to the quotient $\cF^*_x$ of $TX^*$. Then any resolution $\pi:X'\to X$ of the indeterminacies of $\varphi $ possesses the claimed property. In particular, we have an injection (actually: an isomorphism) of sheaves: $\pi^*(\cF^*)/Torsion\to \cQ'$ which induces, for any $m\geq 0$, an injection of $(\pi^*(Sym^m(\cF^*))/Torsion))^{**}$ first in $((Sym^m(\pi^*(\cF^*))/Torsion)^{**}$, and then in $Sym^m(\cQ')$.

 %We deduce from this, for any given $k\geq 0$, and any $m>0$, natural injective maps: $$\pi_m: H^0(X,L^k\otimes (Sym(\cF^*)/Torsion)^{**})\to H^0(X', \pi^*(L^k)\otimes Sym^m(\cQ')).$$

%(c) The equality: $$H^0(X, L^k\otimes (Sym^m(\cF^*)/Torsion)^{**}))=H^0(X_0, L^k\otimes Sym^m(\cF^*)),$$ composed with the previous injection, gives the injection: 
%$$H^0(X_0,L^k\otimes Sym^m(\cF^*))\subset H^0(X',\pi^*(L^k\otimes (Sym^m(\cQ'))))$$ 

%and proves the claim (c)

%Claim (d): Let $P$ be the projectivised bundle of hyperplanes of the bundle $\cQ'$ on $X'$, and let $H$ be its tautological bundle. We consider pairs of non-negative integers $k,m$ such that $0\leq m\leq B.k$. We may assume  that there is an injection of sheaves: $H\subset A$ on $P$, for some very ample line bundle $A$ on $P$.

%We then have, denoting with $M$ the lift to $P$ of $\pi^*(L)$: 

%$h^0(X', \pi^*(L^k)\otimes Sym^m(\cQ'))= h^0(P, M^k\otimes H^m)$

%$\leq h^0(P, M^{k}\otimes A^m)) \leq h^0(P, M^k\otimes A^{B.k})=h^0(P,(M\otimes A^B)^k),$ 

%which gives the claim (c)
\end{proof}

\noindent Thus, Proposition \ref{bounded} is proved as well.
\end{proof}

\noindent Hence, the algebraically criterion is established.
\end{proof}

%%%%%%%%%%%%%%%%%%%%%%%%%%%%%%%%%%%%%%%%%%%%%%%%%%%%%%%%%%%%%%%%%%%%%%%%%%%%%%%%%%%%%%%%%%%%%%%%%%%%%%%%%%%%%%%%%%%%%%%%%%%%%%%%%%%%%%%%%%%%%%%%%%%%%%%%%%%%%%%%%%%%%%%%%%%%%%%%%%%%%%%%%%%%%%%%%%%%%%%%%%%%%%%%%%%%%%%%

\subsection{Rational connectedness} \label{rat}

The following result was proved (by very different arguments) in the case of ample classes in \cite{BMQ} and \cite{KST}. Our proof here is using two main techniques: the existence of the relative rational quotient of a map $p$ and the fact that the projective manifolds whose canonical class is not pseudo-effective are uniruled (actually, this is the unique argument in positive characteristic we need in this paper).

\begin{theorem}\label{rc} Let $X$ be projective smooth manifold, and let $\cF\subset TX$ be a foliation. Assume that 
there exists a movable class $\alpha$ for which $\mu_{\alpha, \rm min}(\cF)>0$. 
Then $\cF$ is an algebraic foliation and its leaves are rationally connected.
\end{theorem}

\begin{proof}
The fact that $\cF$ is an algebraic foliation has been proved. We now treat the last claim of Theorem \ref{rc} using the relative rational quotient.

Let $p: X\dasharrow Z$ be the rational map \eqref{cw} induced by the application $p_0: X\dasharrow \Chow(W)$, with $Z$ a desingularisation of the image of $p_0$. We also consider the \emph{relative rational quotient} of $p$:
\begin{equation}\label{723}   
r: X\dasharrow Y
\end{equation}
This map is constructed in \cite{Ca92} or \cite{KMM} for the absolute version. The existence of the relative version follows from \cite{Ca04}, Appendix. We also have a map $s: Y\dasharrow Z$, such that
$s\circ r= p$. 
\smallskip 

\noindent Assume by contradiction that the fibers of $p$ are not rationally connected, then:
\begin{enumerate}

\item[(a)] $\dim Y> \dim Z$.

\item[(b)] The canonical bundle of the desingularisation of any generic fiber of $s$ is pseudo-effective by \cite{GHS}.

\item[(c)] The generic fibers of $r$ are rationally connected.

\end{enumerate} 
\smallskip

We will consider now regular models of the maps defined above: let $\pi_X: \wh X\to X$ and $\pi_Y: \wh Y\to Y$
be smooth modifications of $X$ and $Y$ respectively, such that the applications 
\begin{equation}\label{4000}
\wh p:= p\circ \pi_X, \quad \wh s:= s\circ \pi_Y
\end{equation}
are regular. We can also assume that there exists a map $\wh r: \wh X\to \wh Y$ such that the equality 
$\wh s\circ \wh r= \wh p$ is preserved.

\noindent Let $\wh \cH:= \ker (d\wh s)$ be the foliation induced by the kernel of the differential of $\wh s$. By formula \eqref{722} combined with Remark \ref{repref} and the property (b) above, we see that 
\begin{equation}\label{800}
\det (\wh \cH^\star)
\end{equation}
is pseudo-effective on $\wh Y$, modulo a divisor which is $\pi_Y$-exceptional. Let $\cH$ be the foliation induced by $\wh\cH$ on $Y$; we  deduce that $\det(\cH^\star) $ is pseudo-effective, by Lemma \ref{l pseffff}.

Let $\wh \cF$ be the foliation induced by $\cF$ on $\wh X$. Then we have a morphism:
\begin{equation}\label{801}
\wh \cF\to \left(\pi_Y\circ \wh r\right)^\star \cH
\end{equation}
and we claim that it is generically surjective. 
The first observation is that the map $\wh \cF\to \wh r^\star \wh\cH$ is well-defined and generically surjective. This is the case
because $\wh X$ and $\wh Y$ are smooth, and for any general enough
$z\in Z$ the map in question is induced by the differential of the map $\wh X_z\to \wh Y_z$. The map $\wh r^\star \wh\cH\to \left(\pi_Y\circ \wh r\right)^\star \cH$ is an isomorphism at the generic point of $\wh Y$.    

We have $\displaystyle \mu_{\pi_X^\star \alpha, \rm min}(\wh\cF)>0$, since $\mu_{\alpha, \rm min }(\cF)>0$, cf. Proposition \ref{pi^*}
and its proof. Hence we infer that
\begin{equation}\label{802}
\mu_{\pi_X^\star \alpha}(\wh r^\star\cH)> 0,
\end{equation}
contradicting the pseudo-effectivity of $\det \cH^\star$.
 \end{proof}
 \medskip

 \begin{rem}\label{logleaf} {\rm The discrepancies $K_{\wh X}-\pi^*(K_X)|_{F}$ of the generic fibre $F$ of $\wh p$ of a `neat model' of the rational fibration $p$ defined by $\cF$ above are of great geometric interest also.}
\end{rem}

\subsection{Pseudo-effectivity of cotangent sheaves of foliations}

 \noindent We establish here a stronger version of Theorem \ref{main} when $\Delta=0$. One of the motivations for this statement is the existence (cf. \cite{CP13}) of foliations with $K_\cF$ pseudo-effective on some projective uniruled manifolds.
 
 \begin{theorem}\label{quotfol}
 Let $X$ be a non-singular projective manifold, and let $\cF\subset \cO(T_X)$ be a foliation on $X$, with $K_{\cF}$ is pseudo-effective. Then, for any positive integer $m\geq 1$, and any coherent, torsion-free sheaf $\cQ$ such that there exists a 
generically surjective map
\begin{equation}\label{0701}
\otimes ^m\cF^\star\to \cQ,
\end{equation} 
$\det \cQ$ is a pseudo-effective line bundle on $X$.
 \end{theorem}
 
 \begin{proof}
 Let $\alpha\in \Mov(X)$ 
be a  movable class; we have to prove that 
\begin{equation}\label{6344}
c_1(\cQ).\alpha\geq 0.
\end{equation}
\smallskip
\noindent By contradiction, assume that the relation \eqref{6344} does not hold. 
Thus $\mu_{\rm \alpha, min}(\wh \otimes ^m\cF^\star)<0$ and by Theorem \ref{tensor, I}
this implies that we have $\mu_{\rm \alpha, min}(\cF^\star)<0$, which in turns shows the inequality
$\mu_{\rm \alpha, max}(\cF)> 0.$

Let $\cG\subset \cF$ the $\alpha$-maximal destabilizing sheaf of $\cF$; then $\cG$ is $\alpha$-semi-stable, and:
\begin{equation}\label{0703}
\mu_{\alpha}(\cG)> 0.
\end{equation} 
It is a simple matter to check that the slope inequalities in Lemma \ref{folrel} below are satisfied, that is to say
 \begin{equation}\label{pente11}
 \mu_{\alpha,min}(\cG)\geq \mu_{\alpha,max}\left(\cF/\cG\right).
 \end{equation} 
 This is a well-known consequence of the maximality of $\cG$, so we only sketch the argument as follows. We consider a sub-sheaf $\ol\cH\subset \cF/\cG$. Then there exists a sub-sheaf $\cH\subset \cF$, containing $\cG$ and inducing $\ol\cH$. Hence we have
$\mu_\alpha(\cH)\leq \mu_\alpha(\cG)$, from which we deduce (after a few standard 
computations which we skip) that $\mu_{\alpha,max}\left(\cF/\cG\right)\leq \mu_\alpha(\cG)$. Now the semi-stability of $\cG$ with respect to $\alpha$ implies the inequality \eqref{pente11}.

In conclusion, $\cG$ is integrable.
Moreover, by Theorem \ref{algebraic} the foliation $\cG$ is algebraic. Let $X_0\subset X$ be the maximal Zariski open set such that $\displaystyle \cG|_{X_0}$ is a vector bundle, and such that the singularities of $\cG$ are contained in the complement 
$X\setminus X_0$ (which has co-dimension greater than two). 

Thus there exists a rational map
$$p: X\dasharrow Z$$
such that $\cG= \ker(dp)$ generically on $X$ --this is the notion of ``algebraicity" we have adopted at the beginning of section 4, cf. Definition \ref{alg}.

\smallskip

\noindent We consider a modification $\pi_X: \wh X\to X$ such that
the composed map $\wh p:= p\circ\pi_X$ is holomorphic.
Since the foliation $\cG$ is the kernel of the differential of the map
$p$ we infer that \emph{the canonical bundle of the fibers of $\wh p$ is not pseudo-effective}, by 
\eqref{0703}.

\smallskip

\noindent Let $\wh \cG$ and $\wh \cF$ be the foliations induced by $\cG$ and $\cF$ on $\wh X$, respectively.
Then we still have $\wh \cG= \ker (d\wh p)$ generically on $\wh X$ and $\wh \cG\subset \wh \cF$.
\medskip

\noindent We shall use the following ``rigidity lemma'' (cf. \cite{AD}, Lemma 6.7 for similar ideas).

\begin{lemma}\label{rigid} Let $\wh\cG\subset \wh\cF$ be two foliations on $X$ Assume that $\wh\cG$ is algebraic, defined generically as $\cG= \ker(d\wh p)$ for a dominant map $\wh p: \wh X\to Z$. There then exists a foliation $\cH$ on $Z$ such that $d\wh p(\wh\cF)= {\wh p}^\star\cH$, generically on $\wh X$.
\end{lemma}

\begin{proof} Let $x_0\in \wh X$ be such that $\wh\cF$ is non-singular at $x_0$, and such that $y_0:= \wh p(x_0)$ is a regular value of $\wh p$. 

Let $\Lambda_0\subset \wh X$ be a germ a submanifold contained in the leaf $\displaystyle L_{x_0}$ of $\cF$ at $x_0$, transverse to $G_0:= \wh p^{-1}(y_0)$,  and such that 
$$\cF_{x_0}= \cG_{x_0}+ T_{\Lambda_0, x_0}$$ 
is a direct sum decomposition.
Next $\wh p(\Lambda_0)$ is a germ of a submanifold $V_0$ of $Z$ at $y_0$, and $W_0:=\wh p^{-1}(V_0)$ is contained, and hence equal to the germ of the leaf $\cF_{x_0}$. Indeed: for each $x\in \Lambda_0$, $\wh p^{-1}(\wh p(x))=\cG(x)\subset \cF(x)$, and $\cF(x)$ thus contains both $\wh p^{-1}(\wh p(x))$, and $\Lambda_0$. 

Since this holds for every $x_0$ having the properties specified above, the lemma follows by analytic continuation. \end{proof}

%We keep $x_0,y_0,V_0,W_0$ as in the lemma \ref{rigid} above. Since $y_0$ is a smooth point of $V_0$, we get: 
%\begin{equation}\label{0713}
%d\wh p(\wh \cF|_{G_0})= T_{y_0}V_0
%\end{equation}
%at the generic point of $G_0$. Notice that $W_0$ is smooth.

%We have a map
%\begin{equation}\label{0715}
%g: W_0\to V_0
%\end{equation}
%induced by the restriction of $\wh p$, and we know from Theorem \ref{rc} that the canonical bundle of the generic fiber of
%$g$ \emph {is not pseudo-effective}.

%Let $\cJ\subset \wh p^\star T_Z$ be the image of the differential of $p$ restricted to the foliation $\cF$, so that we have
%\begin{equation}\label{639}
%0\to \wh\cG\to \wh\cF\to \cJ\to 0
%\end{equation}
%outside a set of codimension at least two. 

%On the other hand, we have a map
%\begin{equation}\label{0716}
%TV_0\to \pi_0^\star(\wh \cF|_{V_0})
%\end{equation}
%which is generically an isomorphism. Then the dual map is injective, so the inverse image of the canonical bundle of $\wh \cF$ injects into the determinant $K_{V_0}$ of $\Omega^1_{V_0}$. Thus $K_{V_0}$ is pseudo-effective when restricted to a generic fiber of $g$. This contradicts the previous observations,
%and concludes the proof.

\noindent As a consequence of Lemma \ref{rigid}, we have the exact sequence
\begin{equation}\label{639}
0\to \wh\cG|_U\to \wh\cF|_{U}\to \cO_U^{\oplus r}\to 0
\end{equation}
where $U:= {\wh p}^{-1}(V)$ and $V$ is a small topological coordinate set centered at a regular value of $\wh p$.

Therefore we have 
\begin{equation}\label{0715}
K_F= K_{\wh\cG}|_F= K_{\wh\cF}|_F
\end{equation}
where $F$ is a generic fiber of $\wh p$. This is however a contradiction, because $K_{\wh\cF}$ is pseudo-effective by hypothesis, whereas $K_F$ is not, by the previous discussion. 
 \end{proof}

\medskip

\noindent  The following corollary is a consequence of the previous rational connectedness statement of Theorem \ref{algebraic}. The claim (1) is a generalization (in the projective case, the result of \cite{Bru} being valid in the compact K\"ahler case as well)   M. Brunella's Theorem (\cite{Bru}). This corollary gives an optimal geometric obstruction to the pseudo-effectivity of the canonical bundle of foliations on projective manifolds.

\begin{corollary}
Let $X$ be a projective manifold, and let $\cF$ be a foliation of rank $0< r < n:= \dim(X)$. Assume 
that $\cF$ is not algebraic. Then:
\begin{enumerate}

\item[(i)] If $r=1$, the bundle $K_{\cF}$ is pseudo-effective.
\smallskip

\item[(ii)] For an arbitrary rank $r$, the bundle $K_{\cF}$ is not pseudo-effective if and only if there exists a non-trivial algebraic foliation $\cG\subset \cF$ such that $\mu_{\alpha,min}(\cG)>0$ for some movable class $\alpha$. 
\end{enumerate}
\end{corollary}

\begin{proof} Claim (i). Since by assumption $\cF$ is not algebraic, Theorem \ref{algebraic} implies, that for each movable class $\alpha$, we have:
\begin{equation}
\mu_{\alpha,min}(\cF)=\mu_{\alpha}(\cF)\leq 0.
\end{equation} 
If the rank $r$ of $\cF$ is equal to one, then this implies (cf. \cite{BDPP}) that $K_{\cF}$ is pseudo-effective, and the point (i) is proved.
\smallskip

\noindent As for the second point, the `if' part can be seen as follows. By dualizing the inclusion $\cG\subset \cF$
we obtain a generically surjective map $\cF^\star\to \cG^\star$. If the canonical bundle of $\cF$ is pseudo-effective, then we infer that $\det \cG^\star$ is pseudo-effective as well, by Theorem \ref{quotfol}. But this contradicts the hypothesis of (ii). 

\noindent 
We thus treat next the `only if' part, and  first remark that we have 
$\mu_{\alpha}(\cF)>0$ for some movable class $\alpha$, again by \cite{BDPP}. 
Next we see that $\cF$ is not semi-stable with respect to $\alpha$. 
This is indeed the case, since if the contrary holds then we have $\mu_{\alpha, min}(\cF)>0$
and therefore $\cF$ would be algebraic. This contradicts our assumptions on $\cF$.

Let $\cG$ be the maximum destabilizing subsheaf of $\cF$ with respect to $\alpha$. It is semi-stable, and has thus strictly positive minimum $\alpha$-slope, i.e.
 \begin{equation}\mu_{\alpha,min}(\cG)>0.
\end{equation} 
As in the proof of Theorem \ref{quotfol}, we have  
 \begin{equation}\label{pente111}
 \mu_{\alpha,min}(\cG)\geq \mu_{\alpha,max}\left(\cF/\cG\right).
 \end{equation} 
\noindent By Lemma \ref{folrel} below, the inequality \eqref{pente11} shows that $\cG$ is a foliation, and is thus algebraic by Theorem \ref{algebraic}. Since $\cF$ is supposed to be non-algebraic, we get: $rank(\cG)<rank(\cF)$ as claimed.
\end{proof} 

\begin{lemma}\label{folrel} Let $\cG\subset\cF\subset TX$ be holomorphic (possibly singular) distributions on $X$ smooth projective connected. Assume that $\cF$ is a foliation, and that for some movable class $\alpha$ we have: $\mu_{\alpha,min}(\cG)>0$ and also: 
\begin{equation}
2.\mu_{\alpha,min}(\cG)>\mu_{\alpha,max}(\cF/\cG)).
\end{equation} Then $\cG$ is a foliation.
\end{lemma}

\begin{proof} The natural composed map $\wedge^2\cG\to TX/\cG\to TX/\cF$ derived from the Lie bracket on $X$ vanishes, since $\cF$ is a foliation, and thus defines a section of $Hom(\wedge^2(\cG)\to (\cF/\cG))$ over $X$. But this vector space vanishes because of the slope conditions. This forces the Lie bracket $\wedge^2\cG\to TX/\cG$ to vanish, as claimed.
\end{proof}

%%%%%%%%%%%%%%%%%%%%%%%%%%%%%%%%%%%%%%%%%%%%%%%%%%%%%%%%%%%%%%%%%%%%%%%%%%%%%%%%%%%%%%%%%%%%%%%%%%%%%%%%%%%%%%%%%%%%%%%%%%%%%%%%%%%%%%%%%%%%%%%%%%%%%%%%%%%%%%%%%%%%%%%%%%%%%%%%%%%%%%%%%%%%%%%%%%%%%%%%%%%%%%%%%%%%%%%%

\subsection{Descent of foliations}
The following consequence of the preceding theorem will mainly be needed in the proof of Theorem \ref{main}. But it may have some interest by itself.

\begin{corollary}\label{descent}
Let $\pi:X'\to X$ be a finite surjective holomorphic map of degree $d>1$ between complex projective manifolds. Let $\alpha$ be a movable class on $X$, and let $\alpha':=\pi^*(\alpha)$ be its pull-back to $X^\prime$. Let $\cF'\subset \pi^*(TX)$ be a torsion-free subsheaf, with $\cF^{\rm sat}$ its saturation in $\pi^*(TX)$. Assume moreover that: \begin{enumerate}

\item[(a)] $\mu_{\alpha',\rm min}(\cF')>0$;
\smallskip

\item[(b)]  $\cF^{\rm sat}=\pi^*(\cF)$ for some subsheaf $\cF\subset TX$.
\end{enumerate}
Then $\mu_{\alpha, \rm min}(\cF)>0$; in particular, if $\cF$ is integrable, then the corresponding foliation is algebraic.
\end{corollary}

\begin{proof} Since $\mu_{\alpha',min}(\cF')>0$ we  deduce that we have
\begin{equation}\label{701}
\mu_{\alpha',\rm min}(\cF^{\rm sat})>0
\end{equation} 
by Remark \ref{sat}.
Let $\cQ$ be a quotient of $\cF$; then $\pi^\star \cQ$ is a quotient of $\pi^\star \cF= \cF^{\rm sat}$. By \eqref{701} above, we  deduce:
\begin{equation}\label{702}
\mu_{\alpha'}(\pi^\star \cQ)>0
\end{equation} 
and this is equivalent with $\displaystyle \mu_{\alpha}(\cQ)> 0$, which is the claim.
The last part of Corollary \ref{descent} follows from the algebraically criteria.
\end{proof}

\begin{rem}\label{rstab} {\rm In general, in the situation of the preceding corollary, if $\cF'$ is $\alpha$-semi-stable, $\cF^s$ does not need to be $\alpha$-semi-stable, as shown by the natural injection of $\cO(1)\oplus \cO(1)$ in $\cO(1)\oplus \cO(2)$ over $\Bbb P^1$.}
\end{rem}

\begin{rem}These results on foliations immediately extend to logarithmic foliations. We show this in the next section, which will also serve as a simplified model for the case of arbitrary smooth `orbifold pairs',  treated below, and for which additional constructions and definitions are required. We added this short section in order to make the application (through corollary \ref{bigness}) to families of canonically polarised manifolds in \S.8.5 below independent from the general `orbifold version'.The proof given here of corollary \ref{bigness} is quite different and shorter from the one given in \cite{S}, which showed that the general orbifold pairs could be avoided. Notice however that, once the foundational material are laid, the continuity method used in Theorem \ref{bignu} gives a much more direct alternative proof of corollary \ref{bigness}. \end{rem}

\bigskip

%%%%%%%%%%%%%%%%%%%%%%%%%%%%%%%%%%%%%%%%%%%%%%%%%%%%%%%%%%%%%%%%%%%%%%%%%%%%%%%%%%%%%%%%%%%%%%%%%%%%%%%%%%%%%%%%%%%%%%%%%%%%%%%%%%%%%%%%%%%%%%%%%%%%%%%%%%%%%%%%%%%%%%%%%%%%%%%%%%%%%%%%%%%%%%%%%%%%%%%%%%%%%%%%%%%%%%%%

%%%%%%%%%%%%%%%%%%%%%%%%%%%%%%%%%%%%%%%%%%%%%%%%%%%%%%%%%%%%%%%%%%%%%%%%%%%%%%%%%%%%%%%%%%%%%%%%%%%%%%%%%%%%%%%%%%%%%%%%%%%%%%%%%%%%%%%%%%%%%%%%%%%%%%%%%%%%%%%%%%%%%%%%%%%%%%%%%%%%%%%%%%%%%%%%%%%%%%%%%%%%%%%%%%%%%%%%

\section{ Orbifold tensor bundles on Kawamata covers}

Let $(X, \Delta)$ be a smooth log canonical pair, written as:
\begin{equation}\label{0717}
\Delta= \sum_{j\in J} c_j D_j=\sum_{j\in J} \big(1-\frac{b_j}{a_j}\big)D_j
\end{equation}
where $J$ is a finite set, and for each $j\in J$ we have 
$0\leq b_j< a_j$ are coprime integers,
and the hypersurfaces $(D_j)$ are snc. If the coefficient $b_j$ is equal to zero, then we agree that the corresponding denominator $a_j$ is equal to 1.

These orbifold pairs $(X, \Delta)$ interpolate between the \emph{compact, or projective case} (i.e. when either $J=\emptyset$) 
and the \emph{logarithmic, or quasi-projective case}, when $b_j= 0$ for all $j\in J$, respectively. In both cases, the 
notions of tangent bundle, cotangent bundles and more generally, of holomorphic tensors are classically defined. They play a fundamental r\^ole in the study of the geometry of (quasi-)projective manifolds.
\smallskip

\noindent We shall introduce the analogous notions corresponding
to an arbitrary orbifold pair $(X, \Delta)$. Unfortunately they can
only be defined
on a suitable ramified cover of $X$ adapted to $(X, \Delta)$.
However, we shall see that they enjoy properties similar to those of the usual ones in the two standard cases (compact, and logarithmic) mentioned above. These properties will turn out to be independent on the cover used to define them. 

\smallskip

The underlying idea for the definition is that the local generators as an $\cO_X$-module of the orbifold cotangent bundle should ``look like'': 
\begin{equation}\label{666}\frac{dz_{1}}{z_{1}^{1-{b_{1}}/{a_{1}}}},
\dots,\frac{dz_{r}}{z_{r}^{1-{b_{r}/{a_r}}}}, 
dz_{n_1+1},\dots,dz_n,\end{equation} 
on some coordinate open set $U\subset X$ where the divisor
$\lceil \Delta\rceil$
is given by $z_1\dots z_r= 0$.
Unlike in the cases mentioned above, these symbols involve multi-valued functions. Nevertheless, we have the identity
\begin{equation}\label{6666}\pi^*\Big(\frac{dz_{}}{z_{}^{1-{b_{}}/{a_{}}}}\Big)=Nw^{Nb/a} \frac{dw}{w},
\end{equation} 
where $z=w^N$, and we see that the right-hand side is an usual logarithmic differential provided that $N/a$ is an integer.
\medskip

\noindent This suggests that in order to construct the tensor bundle corresponding to the pair $(X, \Delta)$, one needs an auxiliary object,
namely a map which ramifies along $D$ with divisible enough order. The formal definition will be given in what follows.

\subsection{Ramified coverings}
We recall in this sub-section a few basic facts concerning global ramified covers associated to an orbifold pair
$(X, \Delta)$, for which a polarization is fixed.
Our reference is \cite{KMM87}, (see also \cite{EV}, \cite{K}).

\begin{defn}\label{ram, I}
  Let $(X, \Delta)$ be an orbifold pair as in \eqref{0717}.
  A ramified cover adapted to $(X, \Delta)$ is by definition a Galois 
  covering $\pi: X_\Delta\to X$ satisfying the following requirements.
\begin{enumerate}

\item[(i)] The variety $X_\Delta$ is non-singular, and the ramification
  order of $\pi$ along each component $D_i$ is equal to $a_i$, i.e.
  $\displaystyle \pi^\star(D_i)= a_i\sum_j D_{ji}$.

\item[(ii)] The support of the divisor $\displaystyle \pi^\star(\Delta)+ \Ram(\pi)$ as well as the branching loci 
$\sum H_j$ of $\pi$ have simple normal crossings.     
\end{enumerate}    
  
\end{defn}  

\noindent Such a map $\pi$ will be referred to as ``Kawamata cover''
in what follows, cf. \cite{KMM}, Theorem 1.1.1.
The properties which will be relevant for us are stated in the following lemma.

\begin{lemma}\label{ram_cov} Let $(X, \Delta)$ be an orbifold pair; then the following assertions are true.
\begin{enumerate}

\item[(a)] The pair $(X, \Delta)$ admits a Kawamata cover.
\smallskip  

\item[(b)] Let $\pi: X_\Delta\to X$ be any Kawamata cover corresponding to $(X, \Delta)$, and let $G$ be the associated Galois group. For any
point $y\in X_\Delta$ there exists an open coordinate set $y\in U$ which is
$G_y$--invariant, and such that the restriction $\pi|_U$ has the following shape  
\begin{equation}\label{0302}
  \pi(w_1,\dots, w_n)= (w_1^{a_1},\dots, w_k^{a_k}, w_{k+1},\dots, w_p,
  w_{p+1}^{m_1},\dots w_n^{m_n} )    
\end{equation}
with respect to co-ordinates $(w_i)$ and $(z_j)$ on $U$ and its image,
respectively.

\end{enumerate}
\end{lemma}

\noindent In the definition above we denote by $G_y$ the isotropy group of $y$. We note that in \eqref{0302} we assume that the divisor
$\lceil\Delta\rceil$ is locally given by the equation $z_1\dots z_k= 0$. Also,
the local hypersurfaces $z_{p+1}=0,\dots, z_n=0$ correspond to the
extra-ramification of $\pi$ --which is in general unavoidable,
but which will not affect us in any way.
\smallskip

As we see from Lemma \ref{ram_cov}, the map $\pi$ can be seen as the
global version of the standard application $w\to z= w^a$ and we will use it in order to define the orbifold co-tangent bundle and its
associated tensor powers.

\subsection{Orbifold tensor bundles}
Let $(X, \Delta)$ be an orbifold pair, and let $\pi: X_\Delta\to X$ be a Kawamata cover. We first introduce here the 
notion of co-tangent bundle associated to $(X, \Delta)$ by following the elegant approach by Y. Miyaoka in \cite{Mi1}.

We denote by $\Omega^1_X\langle\lceil \Delta\rceil\rangle$ the logarithmic tangent bundle associated to 
$(X, \lceil \Delta\rceil)$. Then we have a well-defined residue map
\begin{equation}\label{03022}
\Omega^1_X\langle\lceil \Delta\rceil\rangle\to \bigoplus_i \cO_{\Delta_i}\to 0,
\end{equation}
which induces a map between the $\pi$--inverse images of the sheaves above
\begin{equation}\label{03023}
\pi^\star\Omega^1_X\langle\lceil \Delta\rceil\rangle\to \bigoplus_i \cO_{\pi^\star\Delta_i}\to 0.
\end{equation}
In \eqref{03023} we have used the flatness of $\pi$ in order to identify $\pi^\star\cO_\Delta$ with $\cO_{\pi^\star\Delta}$. By the properties of the map $\pi$, we can write 
$\displaystyle \pi^\star\Delta_i= a_iD_i$ for some Cartier divisor $D_i$ on $X_\Delta$. Therefore, we have a 
quotient map of sheaves
\begin{equation}\label{03024}
\cO_{\pi^\star\Delta_i}\to \cO_{b_i D_i}\to 0
\end{equation}
for every $i$ in our set of indexes.

\noindent All in all, we have a surjective map
\begin{equation}\label{03024}
\pi^\star\Omega^1_X\langle\lceil \Delta\rceil\rangle\to \bigoplus_i \cO_{b_i D_i}\to 0
\end{equation}
and we introduce the following notion.

\begin{defn}The orbifold co-tangent bundle associated to $(X, \Delta)$ is the kernel 
of the map \eqref{03024}. It is a vector bundle of rank $n= \dim (X)$, and it will be denoted in what follows 
by $\pi^\star\Omega^1(X, \Delta)$.
\end{defn}
\noindent Thus, we have the exact sequence
\begin{equation}\label{03024}
0\to \pi^\star\Omega^1(X, \Delta)\to \pi^\star\Omega^1_X\langle\lceil \Delta\rceil\rangle\to \bigoplus_i \cO_{b_i D_i}\to 0.
\end{equation}
\smallskip

\noindent At this point, a few remarks are in order.
\begin{enumerate}

\item[$\bullet$] The bundle $\pi^\star\Omega^1(X, \Delta)$ is $G$-invariant: this is a direct consequence of the definition.
\smallskip

\item[$\bullet$] With respect to the coordinate system in Lemma \ref{ram_cov} (2), the local frame of this bundle is expressed as
$$w_1^{b_1-1}dw_1, \dots, w_{k}^{b_k-1}dw_k, dw_{k+1},\dots, dw_p, w_{p+1}^{m_{p+1}-1}dw_{p+1},\dots, w_{n}^{m_{n}-1}dw_{n}.$$
\smallskip

\item[$\bullet$] The determinant of the bundle $\pi^\star\Omega^1(X, \Delta)$ is quickly computed from the 
sequence \eqref{03024}, 
\begin{equation}\label{03025}
\det\left(\pi^\star\Omega^1(X, \Delta)\right)= \pi^\star(K_X+ \Delta).
\end{equation}

\end{enumerate}
\medskip

\subsection{The tangent bundle and the Lie bracket on orbifolds}
The following definition is natural.

\begin{defn}
The orbifold tangent bundle associated to $(X, \Delta)$ the dual of $\pi^\star\Omega^1(X, \Delta)$. It 
is a $G$-invariant vector bundle, and it will be denoted in the sequel by $\pi^\star T{(X, \Delta)}$.
\end{defn}
\noindent With respect to the coordinates in Lemma \ref{ram_cov}, the local generators of $\pi^\star T{(X, \Delta)}$
can be written as follows
\begin{equation}\label{03026}
w_1^{a_1-b_1}e_1, \dots, w_k^{a_k-b_k}e_k, e_{k+1},\dots, e_n
\end{equation}
where $\displaystyle e_j:= \pi^\star\frac{\partial}{\partial z_j}$ is the local frame of the inverse image $\pi^\star T_X$. 

Remark that the local generators of the orbifold tangent bundle can also be written as follows
\begin{equation}\label{03027}
w_j^{1-b_j}\frac{\partial}{\partial w_j}, \frac{\partial}{\partial w_i}, w_{l}^{1-m_{l}}\frac{\partial}{\partial w_l}
\end{equation}
where $j=1,\dots,k$ as well as $ i=k+1,\dots, p$ and $l=p+1,\dots,n$. In this way, the tangent 
bundle $\pi^\star T{(X, \Delta)}$ looks more like that dual of $\pi^\star\Omega^1(X, \Delta)$.

%\noindent The bundle of vector fields associated to the orbifold pair $(X, \Delta)$ will be denoted by 
%$. It is the dual of $\pi^\star \Omega^1{(X, \Delta)}$. Note that the bundle of holomorphic tensor fields can be seen as a subsheaf
%of the inverse image $\pi^\star TX$. 

%\smallskip

%\noindent The following notion is needed in the sequel.

%\begin{defn} Let $\cF\subset \pi^\star TX$ be a coherent sheaf. We consider a coordinate open set 
%$\Omega_0$ centered at $x_0$ as in \eqref{668}. Let 
%\begin{equation}\label{0791}
%v= \sum_j \varphi_j(w)\pi^\star \frac{\partial}{\partial z_j}
%\end{equation}
%be a local section of $\displaystyle \cF|_{\Omega_0}$. For each $\gamma= 1,\dots, r$
%an element $g\in G_\gamma $ acts on $v$ as follows
%\begin{equation}\label{0792}
%g\cdot v:= \sum_j \varphi_j(g\cdot \xi)\pi^\star \frac{\partial}{\partial z_j}
%\end{equation}
%and the sheaf $\cF$ is said to be $G$-invariant if $g\cdot v$ is a local section of $\displaystyle \cF|_{g^{-1}\Omega_0}$ for any $g$ and $\gamma$.
%For example, the orbifold tangent space $\pi^\star T{(X, \Delta)}$ is $G$-invariant.
%\end{defn}

\medskip

\noindent {\bf Motivation.} Let $\cF\subset T_X$ be a coherent subsheaf. The corresponding Lie bracket 
\begin{equation}\label{03028}
\Lambda^2\cF\to T_X/\cF
\end{equation}
is $\cO_X$-linear, and if this map vanishes identically, then $\cF$ defines a holomorphic foliation.
In the remaining part of this sub-section we will consider the orbifold analogue of these results.

More precisely, let $\cF_\Delta\subset \pi^\star T(X, \Delta)$ be a coherent subsheaf of the orbifold
tangent bundle. Our objective in what follows is twofold: first we show that under some reasonable hypothesis, 
we can construct an $\displaystyle \cO_{X_\Delta}$-linear map
\begin{equation}\label{03029}
\Lambda^2\cF_\Delta\to \pi^\star T(X, \Delta)/\cF_\Delta.
\end{equation}
Then we will show here that if the map \eqref{03029} vanishes identically, then $\cF_\Delta$ is induced by a
holomorphic foliation on $X$ by a very explicit procedure.\qed
\smallskip

\noindent The first step in this direction is
the following statement which permits to recognize the subsheaves of $\pi^\star TX$ which are inverse images of a sheaf on $X$.

\begin{lemma}\label{inverse}\cite{GKKP}, \cite{BC}
 Let $\cF\subset \pi^\star TX$ be a coherent $\displaystyle \cO_{X_\Delta}$-module, which is saturated in the 
inverse image of the tangent sheaf $T_X$. If moreover $\cF$ is $G$-invariant, then there exists a sheaf 
$\cF_X$ of $\cO_X$-modules on $X$ such that 
\begin{equation}\label{658}\cF= \pi^\star (\cF_X).\end{equation}
\end{lemma}
\noindent This result \ref{inverse} is completely proved in the references indicated above. We will only discuss here a particular case, which contains however the main ides of the proof and explains the relevance of the 
hypothesis in a very clear manner.

\begin{proof} 
We assume that the 
local structure of the map $\pi$ near a point $x_0\in X_\Delta$ is given by:

\begin{equation}\label{660}(w_1, w_2,\dots, w_n)\to (w_1^{N}, w_2,\dots, w_n)\end{equation}
and the action of the isotropy group is given by the multiplication with unit roots (of order $N$).

Let $V$ be a local section of $\cF$ defined in a neighborhood of $x_0$. Then we can write
\begin{equation}\label{661}V= \sum_{k=0}^{N-1}w_1^k p_r^*(v_k)\end{equation}
where $v_k$ are local sections of the sheaf $T_X$. This is a consequence of the hypothesis 
$\cF\subset \pi^\star TX$. Since $\cF$ is $G_{x_0}$-invariant, we deduce that $p_r^*(v_k)\in \cF$.
Indeed, if $\mu$ is a primitive $N$-root of unity, then we have
\begin{equation}\label{662}p_r^*(v_k)= \frac{1}{N}\sum_{p=0}^{N-1} \mu^p\cdot V\end{equation}
since $\mu^p\cdot p_r^*(v_k)= p_r^*(v_k)$ and $\mu^p\cdot w_1^k= \mu^{kp}w_1^k$. But then we have
\begin{equation}\label{663}V- p_r^*(v_k)= w_1\sum_{k=1}^{N-1}w_1^{k-1}p_r^*(v_k)\end{equation}
and it is at this point that we are using the fact that
$\cF$ is saturated in the inverse image of $TX$: the relation \eqref{663} above shows that we have
\begin{equation}\label{664}\sum_{k=1}^{N-1}w_1^{k-1}p_r^*(v_k)\in \cF\end{equation}
The same argument as before shows now that $p_r^*(v_1)\in \cF$, and by induction, we deduce that 
$p_r^*(v_k)\in \cF$ for any $k= 0,\dots, N-1$ 

As a conclusion, for any local section $V$ of $\cF$ the components $\pi^*(v_k)$ of the decomposition
\eqref{661} belong to $\cF$. The sheaf $\cF_X$ we seek is generated by the vectors $v_k$ obtained from 
\eqref{661} with $V:= V_i$, a set of local generators of $\cF$ near $x_0$.
\end{proof}

\medskip
%%%%%%%%%%%%%%%%%%%%%%%%%%%%%%%%%%%%%%%%%%
%%%%%%%%%%%%%%%%%%%%%%%%%%%%%%%%%%%%%%%%%%

\noindent Let $\cF_\Delta\subset \pi^\star T(X, \Delta)$ be a coherent $G$-invariant and saturated subsheaf of the orbifold
tangent bundle. We denote by $\cF^s$ the saturation of $\cF_\Delta$ in $\pi^\star T_X$. Then $\cF^s$ is equally $G$-invariant, so by Lemma
\ref{inverse} there exists a subsheaf $\cF_X\subset T_X$ such that 
\begin{equation}\label{030210}
\cF^s= \pi^\star(\cF_X).
\end{equation}
\noindent Let 
\begin{equation}\label{03031}
\Lambda^2\cF_X\to T_X/\cF_X
\end{equation}
be the $\cO_X$-linear map induced by the Lie bracket on $X$. Its $\pi$-inverse image composed with the natural map $\displaystyle \Lambda^2\cF_\Delta\to \Lambda^2\cF^2$
gives the
$\cO_{X_\Delta}$-linear map
\begin{equation}\label{03032}
\Lambda^2\cF_\Delta\to \pi^\star T_X/\cF^s.
\end{equation}
On the other hand, given that $\cF_\Delta$ is saturated inside the orbifold tangent bundle, we have
the equality $\cF_\Delta= \cF^s\cap \pi^\star T(X, \Delta)$. Thus, we infer that the natural map
\begin{equation}\label{03033}
\pi^\star T(X, \Delta)/\cF_\Delta\to \pi^\star T_X/\cF^s
\end{equation}
is injective.
\medskip

\noindent In this setting, we have the following statement, establishing the existence of the Lie bracket for 
orbifolds $(X, \Delta)$.

\begin{proposition}\label{Lie} Let $\cF_X\subset \pi^\star T(X, \Delta)$ be a coherent $G$-invariant and saturated subsheaf of the orbifold
tangent bundle. Then the map \eqref{03032} factors through \eqref{03033}, i.e. we have an 
$\cO_{X_\Delta}$-linear map
\begin{equation}\label{03034}
\Lambda^2\cF_\Delta\to \pi^\star T(X, \Delta)/\cF_\Delta
\end{equation}
\end{proposition}
\noindent Our proof will unfold as follows. Let $U\subset X_\Delta$ be one of the coordinate subsets provided by Lemma \ref{ram_cov}. We first construct lifting of the usual Lie bracket on $X$ 
$$[\cdot,\cdot]_U: \Lambda^2 \pi^\star TX|_U\to \pi^\star TX|_U$$ 
which is only locally defined.
Then we show that the orbifold tangent bundle $\pi^\star T(X, \Delta)$ is closed under this map.

\noindent On the other hand, given any subsheaf $\cG\subset \pi^\star TX$
we show that the map 
$\displaystyle \Lambda^2\cG\to \pi^\star TX/\cG$ induced by
the $\pi$-lifting of the usual Lie bracket on $X$ coincides with the one
given by $[\cdot,\cdot]_U$.
The former 
is globally defined and $\cO_{X_\Delta}$-linear. The proposition follows by a linear combination of these facts.
\medskip

\begin{proof}

\noindent Let $\cL_X$ be the Lie bracket defined on vector fields on $X$: 
\begin{equation}\label{0750}
\cL_X:\Lambda^2 TX\to TX.
\end{equation}

Let $v$ be a local section of the bundle $\pi^\star TX$. 
We chose local coordinates $w = (w_1,\dots, w_n)$ and $z= (z_1,\dots, z_n)$ near $y_0: =\pi(x_0)$ given by Lemma 
\ref{ram_cov} (we remark that the finite group $G$ is not used in the following definition). 
Then the map $\pi:X_\Delta \to X$ is locally written as follows 
\begin{equation}\label{0751}
\pi(w)=(w_1^{a_1}, \dots, w_k^{a_k}, w_{k+1},\dots, w_p, w_{p+1}^{m_{p+1}},\dots, w_{n}^{m_{n}}).
\end{equation}
In order to simplify the notations, let $\displaystyle c_j:= 1-\frac{b_j}{a_j}$ be the coefficient of $D_j$ in $\Delta$;
if the index ``$j$" corresponds to one of the hypersurfaces $H_j$, then we set $c_j:= 0$.

\noindent We can write 
$v$ in a unique manner
\begin{equation}\label{0759}
v= \sum_{I\in \cE_{r, a}}w^I\pi^\star v_I
\end{equation}
where $\cE_{r, a}$ is the set of indices $I=(i_1,\dots, i_k, i_{p+1},\dots i_n)$ such that $0\leq i_j\leq a_j-1$
for $j=1,\dots, k$ and $0\leq i_\alpha\leq m_\alpha-1$ for $i\geq p+1$.
and we use the multi-index notation $w^I:= \prod_jw_1^{i_j}$. The $(v_I)$ above are
local vector fields on $X$.

Then we define 
\begin{equation}\label{0760}
[v_1, v_2]_U:= \sum_{I, J}w^{I+J}\pi^\star\big(\cL_X(v_{1I}\wedge v_{2J})\big).
\end{equation}
We have the following statement, showing that the orbifold tangent bundle is preserved by the local map \eqref{0760} (we thank B.\ Claudon for pointing out a slight inaccuracy in the previous version of it).

\begin{proposition}\label{factor}
 The orbifold tangent space $\pi^*T(X,\Delta)$ is closed under the local 
 bracket $[\cdot,\cdot]_U$.
\end{proposition}

\begin{proof}  We consider the restriction of $[\cdot,\cdot]_U$ to the exterior power of 
the orbifold tangent bundle, composed with the natural projection map
\begin{equation}\label{0751}
[\cdot,\cdot]_{\Delta, U}: \Lambda^2\pi^*T(X,\Delta)\to \pi^*TX/\pi^*T(X,\Delta);
\end{equation}
the claim is that this map is identically zero.

\noindent
By definition, the 
local generators as $\cO_Y$-modules of $\pi^*(TX)$ are
\begin{equation}\label{0752}
\partial_k:=\pi^*\frac{\partial}{\partial z_k}, \quad k= 1,\dots n.
\end{equation}
As already mentioned, the local generators of $\pi^\star T(X,\Delta)$
can be written explicitly as follows
\begin{equation}\label{0753}
w_1^{a_1c_1}\partial_1,\dots, w_k^{a_kc_k}\partial_k, \partial_{k+1},\dots \partial_n.
\end{equation}
Any local function $\displaystyle \varphi\in\cO_{X_{\Delta}}$ 
can be written in an unique manner 
$\varphi (w)= \sum_{I\in \cE_{r, a}}w^I\psi_I(z),$
for some holomorphic functions $(\psi_I)$ defined locally on $X$; in this expression we are using the 
same conventions as in \eqref{0759}.

Let $v=\sum_{j=0}^{n}\varphi_j(w)\partial_j$ be a local section of $\pi^*T(X,\Delta)$; in particular it can be expressed as follows
\begin{equation}\label{0754}
v= \sum_{I\in \cE_{r, a}}w^I\pi^\star \rho_I
\end{equation}
where $\displaystyle \rho_I:= \sum_{j=1}^{n}\psi_{Ij}\frac{\partial}{\partial z_j}$, for each multi-index $I$. 
\smallskip

\noindent The main observation now is that 
\emph{we can assume that the function $\psi_{Ij}$ divisible by $z_j$ provided that the $j^{\rm th}$ index of $I$ satisfies the inequality $0\leq i_j\leq a_jc_j-1$}. This is an immediate consequence of the definitions, and we detail the argument next.

By \eqref{0753} there exists a family of functions $\displaystyle (\mu_j)_{j=1,\dots, n}$ such that we have
\begin{equation}\label{0793}
v= \sum_{j=1}^k\mu_j w^{a_jc_j}\partial_j+ \sum_{j=k+1}^n\mu_j \partial_j;
\end{equation}
we by identifying the coefficients in \eqref{0793}--\eqref{0754}, we obtain
\begin{equation}\label{0794}  
\sum_{I}w^I\psi_{Ij}(z)= \mu_j(w)w_j^{a_jc_j}
\end{equation}
for each $j= 1,\dots, k$. This clearly proves our assertion, since we have $c_j\leq 1$.

%For $j=1$, the equality \eqref{0794} implies that we have
%\begin{equation}\label{0795}
%\sum_{0\leq i_j\leq a_jc_j-1}w_2^{i_2}...w_k^{i_k}\psi_{I_q1}(0, z^\prime)= 0
%\end{equation}
%where for each $q= 0,\dots, a_1c_1-1$ the index $I_q$ appearing in the sum \eqref{0795} is equal to $(q, i_2,\dots, i_k, i_{p+1},\dots, i_n)$ and $z^\prime:= (z_2,\dots, z_n)$. 

\smallskip

Let $\pi^*\cV:= \pi^*\left(T_X\langle\lceil \Delta\rceil\rangle\right)$ be the inverse image of the logarithmic tangent bundle of 
corresponding to the pair $(X, \lceil \Delta\rceil)$. 

\noindent  By relation \eqref{0754} together with the observation above we obtain the decomposition
\begin{equation}\label{0755}
v= \sum_{j=1}^n\sum_{i_j=0}^{a_jc_j-1}w^I\pi^\star V_{Ij}+ 
\sum_{j=1}^n\sum_{i_j= a_jc_j}^{a_j-1}w^I\pi^\star W_{Ij}
\end{equation}
where $V_{Ij}$ above are local sections of $\cV$, and where $W_{Ij}$ are local holomorphic vector fields on $X$, multiple of $\displaystyle \frac{\partial}{\partial z_j}$. In \eqref{0755} we dropped the indexes $i_{p+1},\dots, i_n$ since they are playing no role.
\smallskip

\noindent The proof ends by a case by case analysis.

\begin{enumerate}

\item[(a)] We have 
\begin{equation}\label{0780}
[w^I \pi^\star V_I, w^J\pi^\star V_J]_U= w^{I+J}\pi^\star\big(\cL_X(V_I, V_J)\big)
\end{equation}
so it belongs to $\pi^\star T(X, \Delta)$, given the fact that the logarithmic tangent bundle is stable by the Lie bracket $\cL_X$.
\smallskip

\item[(b)] If $j, r\leq k$ then we have
\begin{equation}\label{0781}
[w_j^{a_jc_j} f\partial_j, w_r^{a_rc_r} g\partial_r]_U= w_j^{a_jc_j}w_r^{a_rc_r}
\pi^\star\big(\cL_X(f\frac{\partial}{\partial z_j}, 
g\frac{\partial}{\partial z_r})\big)
\end{equation}
which clearly belongs to $\pi^\star T(X, \Delta)$.
\smallskip

\item[(c)] If $V_I$ is a local section of $\cV$ and if $r\leq k$ then we have
\begin{equation}\label{0795}
[w^I \pi^\star V_I, w_r^{a_rc_r} g\partial_r]_U= w^{I}w_r^{a_rc_r}
\pi^\star\big(\cL_X(V_I, 
g\frac{\partial}{\partial z_r})\big)
\end{equation}
and, say, if $r\neq 1$ we have $\displaystyle \cL_X(fz_1\frac{\partial}{\partial z_1}, g\frac{\partial}{\partial z_r})= az_1\frac{\partial}{\partial z_1}+ b\frac{\partial}{\partial z_r}$ for some functions $a$ and $b$ whose expression does not matter: the point is that the $\pi$-inverse image of this vector belongs to 
$\pi^\star T(X, \Delta)$ when multiplied with $w_r^{a_rc_r}$. If $r=1$, then no additional explanations are 
required, because of the factor $w_r^{a_rc_r}$. Also, if $l\geq k+1$ we have $\displaystyle \cL_X\big(f\frac{\partial}{\partial z_l}, g\frac{\partial}{\partial z_r}\big)= a\frac{\partial}{\partial z_l}+ b\frac{\partial}{\partial z_r}$
for some (other) functions $a$ and $b$, but the result is the same: the $\pi$-inverse image of this vector belongs to 
$\pi^\star T(X, \Delta)$ when multiplied with $w_r^{a_rc_r}$--remark that there is no vanishing condition imposed for the coefficients $\geq k+1$ in \eqref{0753}.
\end{enumerate}
\noindent The Proposition \ref{factor} is proved. \end{proof}
\medskip

\begin{proposition}\label{lin_2}
Let $\cG\subset \pi^\star TX$ be a coherent subsheaf, such that there exists $\cG_X\subset TX$ with the property that $\cG= \pi^\star \cG_X$. Then the following map induced by $[\cdot,\cdot]_U$
\begin{equation}\label{0782}
\Lambda^2\cG\to \pi^\star TX/\cG
\end{equation}
 coincides with the $\pi$-inverse image of the Lie bracket $\Lambda^2\cG_X\to TX/\cG_X$. It is therefore
 $\cO_{X_\Delta}$-linear and globally defined.
\end{proposition}

\begin{proof}
Let $q_j, \rho_j$ be positive integers, such that $\rho_j\leq a_j-1$. We denote by 
$w^{qa+\rho}:= \prod w_j^{q_ja_j+ \rho_j}.$
The calculation required by the Lemma \ref{lin_2} is very simple,
based on identities of the
following type

\begin{equation}\label{0783}
\begin{split}
w^{qa+\rho} \pi^\star\big( \cL_X(v_1, v_2)\big)& = w^{\rho}\pi^\star \big(z^q\cL_X(v_1, v_2)\big) \\
& = w^{\rho} \pi^\star \big(\cL_X(z^qv_1, v_2)\big)+ \psi\pi^\star v_1
\end{split}
\end{equation}
where $v_j$ are local sections of $\cG_X$, and $\psi$ is a local function on $X_\Delta$. This implies that if $V_1, V_2$ are local sections of $\pi^\star\cG_X$, then we have
\begin{equation}\label{0784}
\pi^\star\cL_X(\varphi(w) V_1\wedge V_2)\equiv \varphi(w)\pi^\star\cL_X(V_1\wedge V_2)
\end{equation}
modulo a vector in $\pi^\star \cG_X= \cG$. This is precisely what we need to prove, given the definition \eqref{0760}.
\end{proof}

\medskip

%\
%\begin{proof}
\noindent We are now ready for the proof of Proposition \ref{Lie}, as follows.
By Proposition \ref{lin_2}, the composed map:
\begin{equation}\label{0986}
\Lambda^2\cF_{\Delta}\to \Lambda^2\cF^{s}\to \pi^\star TX/\cF^{s}
\end{equation}
is $\cO_{X_\Delta}$-linear (recall that $\cF^{s}$ is the saturation of $\cF_\Delta$ in $\pi^\star T_X$). This induces a unique factorisation (cf. Proposition \ref{factor}) through:
\begin{equation}\label{0787}
\Lambda^2\cF_{\Delta}\to \pi^\star T(X, \Delta)/\cF_{\Delta},
\end{equation}
since the maps $\displaystyle \Lambda^2\cF_{\Delta}\to \Lambda^2\cF^{\rm s}$ and $\displaystyle \pi^\star T(X, \Delta)/\cF_{\Delta}\to \pi^\star TX/\cF^{s}$ are both 
\emph{injective}. We are using the fact that 
$\cF_{\Delta}= \cF^{s}\cap \pi^\star T(X, \Delta)$, hence Proposition \ref{Lie} follows from the $\cO_{X_\Delta}$-
linearity of \eqref{0986}.\end{proof}
\medskip

\noindent We have the following consequence of these considerations.

\begin{corollary}\label{flute} Let $\cF_\Delta\subset \pi^\star T(X, \Delta)$ be a coherent subsheaf. Assume that $\cF_\Delta$ is saturated and $G$-invariant. Let $\cF^{s}$  be the saturation of $\cF_\Delta$ in $\pi^\star TX$; by \ref{inverse} we have $\cF^{s}= \pi^\star \cF$. 
We assume moreover that the orbifold Lie bracket \eqref{0787}
 vanishes identically. Then the sheaf $\cF$ defines a holomorphic foliation on $X$.
\end{corollary}

\begin{proof}
By hypothesis, the linear map \eqref{0787} is identically zero, so we obtain the following partial conclusion: \emph{let $v_1$ and $v_2$ be two local sections of 
$\cF_\Delta$; then $[v_1, v_2]_U\in \cF_\Delta$}.

Let $V_1, V_2$ be two local sections of $\cF$: there exists two local holomorphic functions $\varphi_1$ and $\varphi_2$ on $X_\Delta$ such that
\begin{equation}\label{0785}
v_j:= \varphi_j\pi^\star V_j\in \cF_\Delta
\end{equation}
for each $j= 1, 2$.  Hence we have $\displaystyle \varphi_1\varphi_2\pi^\star\cL_X(V_1, V_2)\in \cF_\Delta$
and thus
\begin{equation}\label{0786}
\pi^\star\cL_X(V_1, V_2)\in \cF^s.
\end{equation}
This implies that we have
\begin{equation}\label{0789}
\cL_X(V_1, V_2)\in \cF
\end{equation}
and thus $\cF$ defines a foliation on $X$.
\end{proof}

%%%%%%%%%%%%%%%%%%%%%%%%%%%%%%%%%%%%%%%%%%%%%%%%%%%%%%%%%%%%%%%%%%%%%%%%%%%%%%%%%%%%%%%%%%%%%%%%%%%%%%%%%%%%%%%%%%%%%%%%%%%%%%%%%%%%%%%%%%%%%%%%%%%%%%%%%%%%%%%%%%%%%%%%%%%%%%%%%%%%%%%%%%%%%%%%%%%%%%%%%%%%%%%%

\subsection{The relative canonical bundle of an orbifold fibration} 

 The following results (Theorem \ref{KForb} and Lemma \ref{claim2}) have been shown in \cite{CP13} , Proposition 1.9 and Theorem 2.11, by computing the degree of both sides on complete intersection curves of very ample classes. We present here a different proof.

\begin{theorem}\label{KForb} Let $\cF_{\Delta}\subset \pi^*T(X,\Delta)$ be such that its saturation $\cF_{\Delta}^{sat}$ in $\pi^\star T_X$ is equal to $\pi^\star(\cF)$, where $\cF= \Ker (dp) \subset TX$ is an algebraic foliation induced by the rational fibration
$p:X\dasharrow Z$. If $K_{X}+ \Delta$ is pseudo-effective, then: 
\begin{equation}\label{6004}
\mu_{\pi^\star\beta}(\cF_\Delta)\leq 0,
\end{equation}
for any  movable class $\beta$ on $X$. 
\end{theorem}

\begin{proof} It is based on the 
following two statements of possibly independent interest.

\begin{lemma}\label{claim1} Let $\cF_{\Delta}\subset \pi^*T(X,\Delta)$ be such that $\cF_{\Delta}^{sat}=\pi^*(\cF)$, where $\cF= \Ker (dp) \subset TX$ is an algebraic foliation induced by the rational fibration
$p:X\dasharrow Z$. Then we have 
\begin{equation}\label{7000} 
\det \cF_\Delta^\star= \pi^\star(K_{\cF}+ \Delta^{\rm hor}).
\end{equation}

\end{lemma}
\medskip

\noindent As in Theorem \ref{KForb}, let $p: X\dasharrow Z$ be 
a rational map, and let $\cF:= \Ker(dp)$ be the foliation induced by the kernel of its differential. The following statement holds true; it appears in \cite{Dr} in a slightly different form.

\begin{lemma}\label{claim2}
Let $\pi_X: \wh X\to X$ and let $\pi_Z: \wh Z\to Z$ be a modification of 
$X$ and $Z$, respectively, such that the following properties are satisfied.
\begin{enumerate}

\item[(i)] The induced map $\wh p: \wh X\to \wh Z$ is regular, and let $ E$ be its discriminant divisor.
\smallskip

\item[(ii)] The inverse image 
\begin{equation}\label{6001}
 \wh E:= \wh p^{-1} E
\end{equation}
has normal crossings.
\smallskip

\item[(iii)] Every component of $\wh E$ which is contracted by $\wh p$ is equally contracted by 
$\pi_X$.
\end{enumerate}
Let $\wh \cF$ be the foliation induced by $\cF$ on $\wh X$; then we claim that the following equality 
holds true
\begin{equation}\label{6002}
K_{\wh \cF}= K_{\wh X/\wh Z}- D(\wh p)
\end{equation}
modulo a divisor which is $\pi_X$-exceptional.
\end{lemma}
\medskip

\noindent Before proving these statements, we show that 
they imply Theorem \ref{KForb}. Let $\beta$ be a movable class on $X$; by \eqref{6002}
we have the equality
\begin{equation}\label{7002}
c_1\big(K_{\wh \cF}+ \wh \Delta^{\rm hor}\big). \pi_X^\star \beta= c_1\big(K_{\wh X/\wh Z}+ \wh \Delta^{\rm hor}- D(\wh p)\big). \pi_X^\star \beta
\end{equation}
because $W\cdot \pi_X^\star \beta= 0$ for any $\pi_X$--exceptional divisor
$W$.
Thus we obtain:
\begin{equation}\label{7003}
c_1\big(K_{\wh \cF}+ \wh \Delta^{\rm hor}\big).\pi_X^\star \beta\geq 0
\end{equation}
by Theorem \ref{pseffff}. By (iii) of Lemma \ref{claim2} combined with the equality \eqref{7000} we have 
\begin{equation}\label{7004}
c_1\big(K_{\wh \cF}+ \wh \Delta^{\rm hor}\big). \pi_X^\star \beta= 
c_1\big(\cF_\Delta^\star\big). \pi^\star \beta,
\end{equation}
proving Theorem \ref{KForb}.

\medskip

\noindent $\bullet$ \emph{Proof of {\rm Lemma \ref{claim1}}}
\smallskip

\noindent The equality \eqref{7000} will be shown next to hold by a direct computation in local coordinates.
Let $X_1\subset X$ be a Zariski open set such that the restriction $\displaystyle \cF|_{X_1}$ is a non-singular foliation, and such that we have $X_1\cap D_j\cap D_k= \emptyset$ for each pair of indexes $j\neq k$, cf. \eqref{0717}. We equally assume that $X_1$ does not contain any of the tangency points of $\cF$ with the support of $\Delta^{\rm t}$, i.e. the set of points $z\in \cup D_j$  such that the tangent space of the divisor at $z$
contains $\cF_z$. Here we denote by $\Delta^{\rm t}$ the set of components of $\Delta$ which are not invariant by $\cF$.

\noindent We have
\begin{equation}\label{7005}
\codim_X(X\setminus X_1)\geq 2,
\end{equation}
hence it would be enough to show that \eqref{7000} holds true when restricted to $\pi^{-1}(X_1)$, given that the map 
$\pi$ is finite.
\smallskip

\noindent Let $x_0\in \pi^{-1}(D_1\cap X_1)$ be a point; we have to distinguish between two cases. 

\emph{If $D_1$ is not invariant by $\cF$}, then in particular $D_1$ is horizontal with respect to the map $p$, and moreover we 
can choose the local coordinates $(z_1,\dots, z_n)$ on an open set $U$ containing the point $\pi(x_0)$ such that 
\begin{equation}\label{7006}
D_1\cap U= (z_1= 0),
\end{equation}
and such that $\cF|_{U}$ is generated by 
\begin{equation}\label{7007}
\frac{\partial}{\partial z_1}, \dots, \frac{\partial}{\partial z_q}.
\end{equation}
Near $x_0$ the map $\pi$ is given by $\displaystyle (w_1,\dots, w_n)\to
(w_1^{a_1}, w_2,\dots, w_n)$ and the
intersection $\pi^\star\cF\cap \pi^\star T(X, \Delta)$ is generated by 
\begin{equation}\label{7008}
w_1^{a_1- b_1}\pi^\star\frac{\partial}{\partial z_1}, \pi^\star\frac{\partial}{\partial z_2},\dots, \pi^\star\frac{\partial}{\partial z_q},
\end{equation}
(notations as in Section 5.4) and the formula \eqref{7000} follows.

\emph{If $D_1$ is invariant by $\cF$}, then we first remark that $D_1$
cannot be horizontal with respect to the map $p$ (given that $\cF$ is equal to the kernel of this map generically).
An appropriate choice of coordinates will give
\begin{equation}\label{7009}
D_1\cap U= (z_{q+1}= 0),
\end{equation}
and such that $\cF|_{U}$ is generated by 
\begin{equation}\label{7010}
\frac{\partial}{\partial z_1}, \dots, \frac{\partial}{\partial z_q}
\end{equation}
and the map $\pi$ is $(w_1,\dots, w_q, w_{q+1},\dots, w_n)\to (z_1, \dots z_q, z_{q+1}^{a_{q+1}},\dots, z_n)$. The intersection 
$\pi^\star\cF\cap \pi^\star T(X, \Delta)$ is generated by 
\begin{equation}\label{7011}
\pi^\star\frac{\partial}{\partial z_1}, \pi^\star\frac{\partial}{\partial z_2},\dots, \pi^\star\frac{\partial}{\partial z_q},
\end{equation}
which settles Lemma \ref{claim1} in this second case. 

If the point $x_0$ does not belong to the support of $\pi^{-1}(\Delta)$, then the verification of \eqref{7000} is simpler. Indeed, near such point
the orbifold tangent space coincides with the inverse image of the
tangent bundle of $X$, thus we have $\cF_{\Delta, x_0}^{sat}=
\pi^\star(\cF)_{x_0}$. The formula follows
--we remark that in this case it makes no difference if $\pi$ is ramified at $x_0$
or not. 
\smallskip

\noindent All in all, the lemma is proved.
\qed

\medskip

\noindent $\bullet$ \emph{Proof of {\rm Lemma \ref{claim2}}}
\smallskip

\noindent Let $\cJ\subset \wh p^\star T_{\wh Z}$ be the image of the differential of $\wh p$, so that we have
\begin{equation}\label{639}
0\to \wh \cF\to T\wh X\to \cJ\to 0
\end{equation}
outside a set of codimension at least two. 

Let $x_0$ be a generic point of a component $W$ of $\wh E$ \emph{which is not $\wh p$- exceptional}.
Then
we have a coordinate system centered at $x_0$, say
$(z_1,\dots , z_n)$ with respect to which the map $\wh p$ can be written as follows

\begin{equation}\label{642}
(z_1,\dots , z_n)\to (z_{q+1}, \dots, z_{n-1}, z_n^{k_n})
\end{equation}
where $W= (z_n= 0)$ near $x_0$.

By a direct computation of the differential, we  deduce that $\cJ$ is generated by the vector fields
\begin{equation}\label{645}
\frac{\partial}{\partial t_1}, \frac{\partial}{\partial t_2}, \dots, z_n^{k_n- 1}\frac{\partial}{\partial t_{n-q}}
\end{equation}
near $y_0$. 

\noindent Hence the determinant of $\cJ$ is equal to 
\begin{equation}\label{647}
\det \cJ= -p^\star K_{\wh Z}- \sum_{i}(k_i-1)Y_i
\end{equation}
where the hypersurfaces appearing in the first sum in \eqref{647} correspond to the components of the pre-image of $E\subset \wh Z$ which are
not exceptional with respect to $\wh p$.

Thus, by the sequence \eqref{639} we obtain  
\begin{equation}\label{648}
\det \cF- p^\star K_{\wh Z}- \sum_{i}(k_i-1)\mu^\star Y_i= -K_{\wh X}
\end{equation}
and after rearranging the terms, this can be reformulated as follows
\begin{equation}\label{649}
K_{\wh \cF}= K_{\wh X/\wh Z}- D(\wh p) 
\end{equation}
modulo a divisor which is $\pi_X$-exceptional.
This is what we wanted to prove.
\end{proof}

%%%%%%%%%%%%%%%%%%%%%%%%%%%%%%%%%%%%%%%%%%%%%%%%%%%%%%%%%%%%%%%%%%%%%%%%%%%%%%%%%%%%%%%%%%%%%%%%%%%%%%%%%%%%%%%%%%%%%%%%%%%%%%%%%%%%%%%%%%%%%%%%%%%%%%%%%%%%%%%%%%%%%%%%%%%%%%%%%%%%%%%%%%%%%%%%%%%%%%%%%%%%%%%%

%%%%%%%%%%%%%%%%%%%%%%%%%%%%%%%%%%%%%%%%%%%%%%%%%%%%%%%%%%%%%%%%%%%%%%%%%%%%%%%%%%%%%%%%%%%%%%%%%%%%%%%%%%%%%%%%%%%%%%%%%%%%%%%%%%%%%%%%%%%%%%%%%%%%%%%%%%%%%%%%%%%%%%%%%%%%%%%%%%%%%%%%%%%%%%%%%%%%%%%%%%%%%%%%%%%%%%%%

\section{Proof of Theorems \ref{rco} and \ref{main}}

\subsection{Proof of Theorem \ref{rco}.}

We recall here some notions. Let $(X,\Delta)$ be a smooth projective log-canonical orbifold pair. Let $\pi:X_\Delta\to X$ be a Kawamata cover adapted to $(X,\Delta)$, and let $\pi^*\Omega^1(X,\Delta)$ and $\pi^*T(X,\Delta)$ be its cotangent and tangent bundles, respectively. Let further 
$\cF_{\Delta}\subset \pi^*T(X,\Delta)$ be a coherent, saturated subsheaf. We denote by $\cF_{\Delta}^{sat}$
the saturation of $\cF_{\Delta}$ in $\pi^\star T_X$. We assume that we have 
$\cF_{\Delta}^{sat}=\pi^*(\cF_X)$ for some uniquely determined distribution $\cF_X\subset TX$ on $X$.
\smallskip

\noindent Denote by $\Psi:\wedge^2\cF_{\Delta}\to \pi^*T(X,\Delta)/\cF_{\Delta}$ the associated orbifold Lie bracket (cf. Corollary \ref{Lie}).

\begin{defn}\label{ofol} We say that $\cF_{\Delta}$ is a foliation on $(X,\Delta)$ if the above map $\Psi$ vanishes identically.
\end{defn}

\begin{rem}\label{folo}{\rm 
    We assume that
    $$2.\mu_{\pi^*\alpha,min}(\cF_{\Delta})>\mu_{\pi^*\alpha,max}(\pi^*T(X,\Delta))/\cF_{\Delta}$$ for some movable class $\alpha$ on $X$. Then $\pi^*\cL_X$ vanishes identically, by the usual slope considerations.} \end{rem}

\begin{rem}{\rm
If $\mu_{\pi^*\alpha,max}(\pi^*T(X,\Delta))>0$ then we obtain a foliation $\cF_{\Delta}$ by choosing appropriate pieces of the Harder-Narasimhan filtration of $\pi^*T(X,\Delta)$.} 
\end{rem}

\noindent Recall the statement of Theorem \ref{rco}:

\begin{theorem}\label{rco'} Let $\cF_{\Delta}\subset \pi^*(T(X,\Delta))$ be such that: $\mu_{\pi^\star\alpha,min}(\cF_{\Delta})>0$, and such that
  $2.\mu_{\alpha,min}(\cF_{\Delta})> \mu_{\alpha,max}(\pi^*(T(X,\Delta)/\cF_{\Delta})$. The saturation of $\cF_{\Delta}$ in $\pi^*(T_X)$ then defines an algebraic foliation $\cF_X$ on $X$ such that the restriction of $K_X+\Delta$ to the closure $F$ of the generic leaf of $\cF$ is not pseudo-effective. \end{theorem}

\begin{proof} By Remark \ref{folo} and Corollary \ref{flute} we infer
  that $\cF_X$ is a foliation on $X$. Moreover we have
  $$\mu_{\alpha,min}(\cF_X)\geq \mu_{\pi^*\alpha,min}(\cF_{\Delta})>0.$$
  Next, by Lemma \ref{claim1} and \ref{claim2} we see that
  $K_{X/Z}+\Delta$ is not pseudo-effective (we use the same notations as in these statements). We obtain Theorem \ref{rco'} as a consequence of
  Remark \ref{repref} (and the references therein).\end{proof}

\subsection{Proof of Theorem \ref{main}}

\begin{proof} Let $(X, \Delta)$ be a smooth lc pair as in Theorem \ref{main}, and let $\pi: X_\Delta\to X$
be a ramified cover associated to this pair, given by Lemma \ref{ram_cov}. Consider any quotient $\otimes ^m\pi^\star\Omega^1{(X, \Delta)}\to \cQ\to 0.$

By contradiction, assume that $c_1(\cQ). \pi^\star\alpha<0$ for some movable class $\alpha$ on $X$. Then we have
$\mu_{\pi^*\alpha,min}(\otimes ^m\pi^\star \Omega^1{(X, \Delta)})<0$.
By Theorem \ref{tensor, I} in Section 2 the inequality
$\mu_{\pi^*\alpha, max}(\pi^\star T{(X, \Delta)})> 0$ is equally satisfied.\smallskip

Let $0\to \cF_\Delta\to \pi^\star T{(X, \Delta)}$ be the maximal $\pi^\star\alpha$ destabilizing subsheaf of the orbifold tangent bundle. 
The maximality of $\cF_\Delta$ induces a few important properties: 
it is $\pi^\star\alpha$-semistable, $G$-invariant and saturated in $\pi^\star T{(X, \Delta)}$.

Moreover, we have: $\mu_{\pi^\star\alpha}(\cF_\Delta)> 0,$ and 
$$2\mu_{\pi^\star\alpha}(\cF_\Delta)>\mu_{\pi^\star\alpha, max}(\pi^*T(X,\Delta)/\cF_\Delta).$$ The conclusion then follows by combining
Remark \ref{folo}, Corollary \ref{flute} and Theorem \ref{KForb}. Indeed,
the saturation of $\cF_\Delta$ in the $\pi$-inverse image of $T_X$ is the inverse image of a foliation $\cF_X$ on $X$. It turns out that $\cF_X$ is algebraic, and the contradiction follows by applying \ref{KForb}. 
\end{proof}

\section{Birational stability of the orbifold cotangent bundle}

\noindent In this section we show the \emph{birational stability} of $\Omega^1(X,\Delta)$ if $K_X+\Delta$ is pseudo-effective, in the sense that, the numerical dimension of any sub-line bundle $\cL$ of $\otimes^m(\Omega^1(X,\Delta))$ is bounded from above by the numerical dimension of $K_X+\Delta$, this for any $m\geq 0$. This term was introduced in \cite{Ca07} to express the fact that the positivity of the subsheaves of the cotangent bundles (measured in terms of sections rather than slopes) is at most the same as for the cotangent bundle itself.

If the bundle $\cL$ is big, then it turns out that the assumption \emph{$K_X+\Delta$ pseudo-effective} can be dropped. This result was obtained in \cite{CP13} by delicate arguments using crucially \cite{BCHM}. The strengthening from `generic semi-positive' to `pseudo-effective' obtained here permits to give below a short and obvious argument, without using \cite{BCHM}.

%%%%%%%%%%%%%%%%%%%%%%%%%%%%%%%%%%%%%%%%%%%%%%%%%%%%%%%%%%%%%%%%%%%%%%%%%%%%%%%%%%%%%%%%%%%%%%%%%%%%%%%%%%%%%%%%%%%%%%%%%%%%%%%%%%%%%%%%%%%%%%%%%%%%%%%%%%%%%%%%%%%%%%%%%%%%%%

\subsection{Numerical dimension}

Let $X$ be smooth and projective, and let $\cL$ be a line-bundle (or $\bQ$-line-bundle) on $X$. Let $A$ be any sufficiently ample line bundle on $X$. Recall from \cite{Nob}:

\begin{defn}\label{nu} The numerical dimension $\nu(X,\cL)$ of $\cL$ is defined by:$$\nu(X, \cL):= \max\{k\in \ZZ\vert \lim\sup_{p\to\infty} \frac{h^0(X,p.\cL+A)}{p^k}<+\infty, \}.$$
\end{defn}

We have the following properties: 

1.1. $\nu(X,\cL)\in \{-\infty,0,1,...,n\}$, and $\nu(X,\cL)\geq \kappa(X,\cL)$. 

1.2. $\nu(X,\cL)\geq 0$ if and only if $\cL$ is pseudo-effective. See \cite{Nob},\S5, Lemma 1.4.

1.3. $\nu(X,k.\cL)=\nu(X,\cL)$, for any $k>0$.

1.4. $\nu(X,k.\cL+P)\geq \nu(X,\cL)$ if $P$ is a pseudo-effective $\bQ$-line bundle.

1.5. If $\nu(X,\cL)=n$, then $\kappa(X,\cL)=n$ (ie: $\cL$ is `big').

\

This variant of the Iitaka-Moishezon dimension of $\bQ$-line bundles permits, when applied to adjoint line-bundles, to turn conjectures of Abundance-type into theorems. We define now:

\begin{defn}\label{nu+} Let $(X,\Delta)$ be a projective log-canonical pair with $X$ smooth and $\Delta$ supported on an snc divisor. Let $\pi:X_{\Delta}\to X$ be an adapted cover.
  Let
  $$\nu^+(X,\Delta):= \max\{\nu(X,\cL)\vert \exists \hbox{ $m$ such that }\pi^*(\cL)\subset
  \pi^*\otimes^{m}\Omega^1(X,\Delta)\}.$$
We define, as usual: $\nu(X,\Delta):=\nu(X,K_X+\Delta)$.

\end{defn}

\noindent We obviously have the following properties: 

P.1. $\nu^+(X,\Delta)\geq \nu(X,\Delta)\geq \kappa(X,\Delta)$. We shall show below the equality: $\nu^+(X,\Delta)= \nu(X,\Delta)$, when $K_X+\Delta$ is pseudo-effective.

P.2. When $\Delta=0$, we thus have, if $K_X$ is pseudo-efective: $$\kappa(X)\leq \kappa^+(X)\leq \nu^+(X)=\nu(X),$$ where $\kappa^+(X)$ was defined in \cite{Ca95} as:$$\kappa^+(X)=max\{\kappa(X,det(F))\vert F\subset \Omega^p(X),p>0\}.$$

P.3. Let $X=\bP^d\times Y$, with $dim(Y)=n-d<n$, and $K_Y$ pseudo-effective. Then $\nu(X)=-\infty$, while $\nu^+(X)=\nu(Y)\geq 0$. These examples show that the restriction $K_X$ pseudo-effective is needed, and explain why the this condition can be dropped when $\nu^+(X)=n$.

P.4. Let $r_X:X\to R_X$ be the `rational quotient' of $X$ (called also its `MRC-fibration'). It has rationally connected fibres and non-uniruled base (by \cite{GHS}). One can easily show that $\nu^+(X)=\nu(R_X)$.

%%%%%%%%%%%%%%%%%%%%%%%%%%%%%%%%%%%%%%%%%%%%%%%%%%%%%%%%%%%%%%%%%%%%%%%%%%%%%%%%%%%%%%%%%%%%%%%%%%%%%%%%%%%%%%%%%%%%%%%%%%%%%%%%%%%%%%%%%%%%%%%%%%%%%%%%%%%%%%%%%%%%%%%%%%%%%%

\subsection{Birational stability of orbifold cotangent bundles}

\begin{theorem}\label{nu} Let $(X, \Delta)$ to be smooth projective orbifold pair, such that $K_X+ \Delta$ is pseudo-effective. Then $\nu^+(X,\Delta)=\nu(X,\Delta)$. In other words:

Let $\cL$ be a line bundle on $X$, together with a non-trivial morphism $\pi^\star \cL\to \otimes^{m}\pi^\star\Omega^1(X, \Delta)$. Then: $\nu(X,\cL)\leq \nu(X,K_X+ \Delta)$.
\end{theorem}

\begin{proof} Let $Q:=\otimes^{m}\pi^\star \Omega^1(X, \Delta)/\pi^\star \cL$ the quotient sheaf. Since $\det(Q)=q_m\pi^\star(K_X+ \Delta)-\pi^\star \cL$, where the constant $q_m= n^{m-1}$ only depends on $m$ and the dimension $n$ of $X$, we have:
  $$q_m.(K_X+\Delta)=\cL+P.$$
Here $P$ is a $\bQ$-line bundle on $X$, whose $\pi$-inverse image is equal to $\det(Q)$.

The bundle $\det(Q)$ is non-negative when evaluated on any inverse image of any  movable class on $X$, by Theorem \ref{main}. Thus
$P$ is pseudo-effective, and therefore
$\nu(K_X+\Delta)\geq \nu(\cL)$ by property P.4 above.
\end{proof}

\medskip
\begin{rem} \label{rbignu} {\rm Assume that the line bundle $\cL$ in Theorem \ref{nu} is big.
Then $K_X+\Delta$ is big if pseudo-effective. Indeed we  deduce that $\nu(X,K_X+\Delta)= n$, and then the conclusion follows from the property P.5 in subsection 7.1 (or, without it, from the fact that the sum of a pseudo-effective and of a big $\Bbb Q$-line bundle is big). We shall remove the hypothesis ``$K_X+\Delta$ pseudo-effective'' in the next subsection.}
\end{rem}

\begin{rem}\label{tajin} {\rm The following observation has been communicated by Behrouz Taji: if $X$ is smooth projective, $n$-dimensional, and if $\nu(X)=0$, with $\chi(X,\cO_X)\neq 0$, then $\pi_1(X)$ is finite of cardinality at most $2^{n-1}$. 

To see this, just apply \cite{Ca95}, theorem 4.1, which says that $X$ has a finite fundamental group of cardinality at most $2^{n-1}$ if $\kappa^+(X)\leq 0$. But now, observe that: $\kappa^+(X)\leq \nu^+(X)=\nu(X)=0$. The last equality holds by Theorem \ref{nu}, since $K_X$ is pseudo-effective if $\nu(X)=0$. Notice that $\nu(X)=0$ implies $\kappa(X)=0$, by \cite{K'}, and that the converse is conjecturally true by Abundance. It was conjectured in \cite{Ca95} that $\kappa^+(X)=\kappa(X)$ if $\kappa(X)\geq 0$.}
\end{rem}

%%%%%%%%%%%%%%%%%%%%%%%%%%%%%%%%%%%%%%%%%%%%%%%%%%%%%%%%%%%%%%%%%%%%%%%%%%%%%%%%%%%%%%%%%%%%%%%%%%%%%%%%%%%%%%%%%%%%%%%%%%%%%%%%%%%%%%%%%%%%%%%%%%%%%%%%%%%%%%%%%%%%%%%%%%%%%%%%%%%%%%%%%%%%%%%%%%%%%%%%%%%%%%%%%%%%%%%%%%%%%%%%%%%%%%%%

\subsection{Criteria for pseudoeffectivity and log-general type} 

\begin{theorem}\label{rem1} Let $(X, \Delta)$ be a smooth orbifold log-canonical pair, 
and let $L$ be a pseudo-effective line bundle on $X$. We assume that there exists a non-zero map 
\begin{equation}\label{mp1}
\pi^\star L\to \pi^\star\Omega^{\otimes m}(X, \Delta)\otimes \pi^\star K_{(X, \Delta)}^{\otimes p}
\end{equation}
for integers $m\geq 0$ and $p>0$. Then $K_X+ \Delta$ is pseudo-effective. (The converse is obvious, taking $m=0,p=1$).
\end{theorem}

\begin{proof} Let $H$ be an ample line bundle on $X$. 
Let $t_{\rm min}$ be the minimum of the positive real numbers $t$ such that
\begin{equation}
K_X+ \Delta+ tA
\end{equation}
is pseudo-effective. The existence of $t_{\rm min}$ is guaranteed by the fact that the pseudo-effective cone is closed. 

We claim that $t_{\rm min}= 0$. If not, let $(t_k)\subset \bQ_+$ be a decreasing sequence of (positive) rational numbers converging to $t_{\rm min}$. Since $A$ is ample, there exists a smooth $\bQ$-divisor $H$ in the linear system $|A|$ such that the orbifold   
\begin{equation}
\left(X, \Delta+ t_kH\right)
\end{equation}
is log-smooth and log-canonical for each $k\geq 1$. If we denote by 
$\pi_k$ the corresponding ramified cover, then the map \eqref{mp1}
induces an injective morphisme of sheaves 
\begin{equation}\label{mp2}
\pi_k^\star L\otimes \pi_k^\star K_{(X, \Delta)}^{-p}\to \pi_k^\star\Omega^m(X, \Delta+ t_kH)
\end{equation}
and let $Q_k$ be the co-kernel of \eqref{mp2}. As in the proof of \ref{bignu} we infer that we have
\begin{equation}\label{mp3}
c(m,n)(K_X+ \Delta+t_kH)= L- p(K_X+ \Delta) + P_k
\end{equation}
where $P_k$ is pseudo-effective. But this implies that 
\begin{equation}\label{mp4}
K_X+ \Delta+t_k\frac{c(m,n)}{p+ c(m,n)}H
\end{equation}
is pseudo-effective, for each value of the parameter $k$. 

On the other hand, there exists 
$k_0\gg 0$ such that 
$$\displaystyle t_{k_0}\frac{c(m,n)}{p+ c(m,n)}
< t_{\rm min}$$ 
since we have assumed that $t_{\rm min}> 0$ is a strictly positive number. Combined with the fact that the $\bQ$-bundle in 
\eqref{mp4} is pseudo-effective for $k:= k_0$, this is in contradiction with the choice of $t_{\rm min}$.
\end{proof}
\medskip

\begin{rem} When $m>0,p=0$, the above situation occurs with $K_{(X, \Delta)}$ either pseudo-effective, or not pseudo-effective, as one sees by considering $X=\Bbb P^k\times Z_{n-k}$, for $0\leq k<n$, if $\Delta=0, K_Z$ pseudo-effective.
\end{rem}

\noindent When $p<0$ instead, we get a lower bound for the existence of $L$, by the same method.

\begin{theorem}\label{rem2}
Let $(X, \Delta)$ be a smooth orbifold log-canonical pair such that 
$K_X+ \Delta$ is pseudo-effective, but not numerically trivial. 

If $p>n^{m-1}$ is an integer,
every map $\pi^\star L\to \pi^\star\Omega^{\otimes m}(X, \Delta)\otimes \pi^\star K_{(X, \Delta)}^{-\otimes p}$ vanishes, for any pseudo-effective line bundle $L$ on $X$. 

In particular: $h^0\left(X_\Delta, \pi^\star\Omega^{\otimes m}(X, \Delta)\otimes \pi^\star K_{(X, \Delta)}^{-\otimes p}\right)= 0$, if $p>n^{m-1}$.
\end{theorem}

\noindent 

\begin{proof}Let a non-zero map from $\pi^*(L)$ to $\pi^\star\Omega^{\otimes m}(X, \Delta)\otimes \pi^\star K_{(X, \Delta)}^{-\otimes p}$ be given. The same arguments as above show that $L+P+p.(K_{(X, \Delta)})=n^{m-1}.K_{(X, \Delta)}$, for some pseudo-effective $P$. Let now $\alpha\in Mov(X)$ be such that $(K_{(X, \Delta)}).\alpha>0$ (this is here that the numerical non-triviality of $K_{(X, \Delta)}$ is used). We get: $(n^{m-1}-p).(K_{(X, \Delta)}.\alpha)\geq 0$, and the conclusion by dividing by $K_{(X, \Delta)}.\alpha$. \end{proof}

\begin{rem} {\rm The trivial example of an Abelian variety $X$ together with $\Delta=0$ shows that for any $(m,p)\in \Bbb Z^{\oplus 2}$ the conclusion may fail when $K_{(X, \Delta)}$ is trivial. It however holds for any blow-up of these $X's$. This example illustrates again the fact that our results are stable by blow-ups, but not necessarily by contractions.}\end{rem} 

\begin{rem}{\rm Also, we note that this statement is considerably weaker than the version obtained in \cite{CP14}, where the same conclusion is obtained under the assumption that $p>m$. However, the technical tools needed in \cite{CP14}
for the proof of this sharper result are much more involved than the present arguments.}\end{rem}

\begin{theorem}\label{bignu} {\rm (\cite{CP13})} Let $(X, \Delta)$ be a smooth log-canonical pair, together with a big line bundle $\cL\to X$ which admits a non-trivial morphism 
\begin{equation}\label{0802}
\pi^\star \cL\to \otimes^{m}\pi^\star \Omega^1(X, \Delta).
\end{equation} 
Then $K_X+ \Delta$ is big.
\end{theorem}

\begin{proof}  One proof is essentially the same as the one used for Theorem \ref{rem1} above, and also as the one used in \cite{CP13}, and of Theorem 2.3 in \cite{CPe}, which deals with the case $\Delta=0$. The statement can also however be directly deduced from the preceding Theorem \ref{rem1} by exactly the same extremely short argument used to deduce Corollary \ref{bigness} from Theorem \ref{logomegapseff}, and to which we refer.
\end{proof}

\medskip

%%%%%%%%%%%%%%%%%%%%%%%%%%%%%%%%%%%%%%%%%%%%%%%%%%%%%%%%%%%%%%%%%%%%%%%%%%%%%%%%%%%%%%%%%%%%%%%%%%%%%%%%%%%%%%%%%%%%%%%%%%%%%%%%%%%%%%%%%%%%%%%%%%%%%%%%%%%%%%%%%%%%%%%%%%%%%%

\subsection{Cases $-(K_X+\Delta)$ either ample, or numerically trivial.}

We give here a strengthened form of a result in \cite{CP13}. The proof is exactly the same as the one of theorem \ref{nu} above, so we just state the result.

\begin{theorem}\label{K=0} Let $(X,\Delta)$ be smooth, projective and log-canonical. Assume that $K_X+\Delta\equiv 0$. 

Then $\mu_{\pi^*(\alpha),max}(\pi^*(\Omega^1(X,\Delta))\leq 0$. 

Let $\pi^*L\to \otimes^m(\pi^*\Omega^1(X,\Delta)), m>0$ be a non-zero sheaf morphism, for some line bundle $L$ on $X$. 

 Then: $-L$ is pseudo-effective. In particular: $\kappa(X,L)\leq 0$.\end{theorem}

 \begin{rem} {\rm It is proved in \cite{CC}, Theorem 4.6, that if $c_1(K_X+\Delta)=0$ (resp. if $-(K_X+\Delta)$ is ample), and if the coefficients of $\Delta$ are `standard' (ie: of the form $c_j=(1-\frac{1}{m_j})$, with $m_j>0$ integer), the orbifold fundamental group $\pi_1(X,\Delta)$ is almost abelian (resp. finite).} 
 \end{rem}
\medskip

\noindent By combining these ideas with adjacent techniques, the following vanishing result is established in \cite{Camp16}, using algebro-geometric arguments in characteristic $0$ only.

\begin{theorem}\label {van2}\cite{Camp16} Let $(X,\Delta)$ be a smooth projective klt orbifold pair. Assume that $-(K_X+\Delta)$ is ample. Let $\pi:X_\Delta\to X$ be a Kawamata cover adapted to $\Delta$. Then, for any $m>0$ and any line bundle $L'\equiv 0$ on $X_\Delta$, we have: $H^0(X_\Delta,\otimes^m(\pi^*\Omega^1(X,\Delta))\otimes L')=0$. Moreover $\pi_1(X)=\{1\}$.
\end{theorem}

%%%%%%%%%%%%%%%%%%%%%%%%%%%%%%%%%%%%%%%%%%%%%%%%%%%%%%%%%%%%%%%%%%%%%%%%%%%%%%%%%%%%%%%%%%%%%%%%%%%%%%%%%%%%%%%%%%%%%%%%%%%%%%%%%%%%%%%%%%%%%%%%%%%%%%%%%%%%%%%%%%%%%%%%%%%%%%

%%%%%%%%%%%%%%%%%%%%%%%%%%%%%%%%%%%%%%%%%%%%%%%%%%%%%%%%%%%%%%%%%%%%%%%%%%%%%%%%%%%%%%%%%%%%%%%%%%%%%%%%%%%%%%%%%%%%%%%%%%%%%%%%%%%%%%%%%%%%%%%%%%%%%%%%%%%%%%%%%%%%%%%%%%%%%%%%%%%%%%%%%%%%%%%%%%%%%%%%%%%%%%%%%%%%%%%%%%%%%%%%%%%%%%%%%%%%%%%%%%%%%%%%%%%%%%%%%%%%%

\section{Variation and Positivity for quasi-projective families.}

We mention here an application of Theorem \ref{bignu} in the theory of moduli. An extremely simplified proof of Theorem \ref{bignu} is presented in the next section.

Let $f:V\to B$ a projective submersion with connected fibres between two quasi-projective connected manifolds $V ,B$. The `variation' $Var(f)\in \{0,..., d:=dim(B)\}$ of $f$ is the rank of the Kodaira-Spencer map $\kappa\sigma: TB\to R^1f_*(TV/B)$ at the generic point of $B$. Thus $Var(f)=0$ if and only if $f$ is isotrivial.

Let $\bar B$ be any `good' smooth projective compatification of $B$, such that $D=\bar B-B$ is an snc divisor.

\

\noindent The following result was conjectured by E. Viehweg, generalizing a former hyperbolicity conjecture of I.R. Shafarevich. Special cases where obtained previously by \cite{KeKo}, \cite{JK}, \cite{Pata}.

\begin{theorem}\label{shafhyp} Let $f:V\to B$ be as above. Assume that the fibres of $f$ all have an ample canonical bundle and that $Var(f)=dim(B)$. Then the base $B$ is of log-general type,
i.e.
$$\kappa(\bar B, K_{\bar B}+D)= \dim (B).$$
\end{theorem}

\begin{proof} In \cite{VZ}, Viehweg-Zuo have shown that, in this situation, for some $m>0$, there exists a big sub-line bundle $\cL$ of $Sym^m(\Omega^1_X(Log(D))$. From Theorem \ref{bignu} we deduce that $K_{\bar B}+D$ is big. 
\end{proof}

Partial generalisations have been obtained in \cite{PoS} and \cite{Taj}, also using Theorem \ref{bignu} and variants of the Viehweg-Zuo sheaf.

In \cite{PoS}, it is shown that, for $f:V\to B$ as above, $\bar B$ is of Log-general type if the fibres of $f$ are of general type, and if the (birational) variation is maximal (equal to $d$). 

In \cite{Taj}, the `isotriviality conjecture' formulated in \cite{Ca07} is solved. This conjecture says that if $B$ is `special', and if the fibres of $f$ are canonically polarised, then $f$ is isotrivial. 

Recall that $B$ being `special' means that $\kappa(\bar B, \cL)<p$, for any $p>0$ and any $\cL\subset \Omega^p_{\bar B}(Log(D))$. (Very) particular cases of `special' quasi-projective manifolds are the ones such that $\kappa(\bar B, K_{\bar B}+D)=0$ for some (or any) good projective compactification $\bar B$ of $B$. We refer to \cite{Ca07} for more details on `specialness' and structure results.

When $d=1$, the only `special' quasi-projective curves are: $\bP^1, \bC, \bC^*$, and $E$, any elliptic curve. I.R. Shafarevich originally formulated his `hyperbolicity conjecture' as the isotriviality of smooth families of curves of genus at least $2$ parametrised by a `special' quasi-projective curve. Remark that quasi-projective curves are `special' if and only if non-hyperbolic. In higher dimensions, there are (lots of) `special' quasi-projective manifolds $\bar B$ of all possible log-Kodaira dimensions less than $d:=dim(B)$.

The preceding results suggest the more general `isotriviality question':

{\bf Question:} Let $f:V\to B$ be as above\footnote{One may even assume only that $f$ be `quasi-submersive', meaning that the reduction of each of its fibres is smooth. The conclusion should then hold by replacing $B$ with the `orbifold base' of $f$, in the sense of \cite{Ca04}. This conjecture was formulated in this form in \cite{AC} when the reduced fibres of $f$ have a semi-ample canonical bundle.} Assume that the fibres of $f$ have a pseudo-effective canonical bundle. If $B$ is `special', is then $f$ is birationally isotrivial? If the birational variation of $f$ is maximal, is then $B$ is of log-general type?

\

The question is also interesting when $f$ has Fano fibres. A. Kuznetsov in \cite{Kuz}, mentions that `Gushel-Mukai' manifolds (complete intersections in $Gr(2,5)$ of Pl\"ucker hyperplanes and one hyperquadric) provide non-isotrivial families of Fano threefolds with Picard number $1$ parametrised by a smooth projective surface. These families are, however, birationally isotrivial.

\subsection{Criteria for pseudoeffectivity and bigness of `purely' logarithmic cotangent bundles.} 
This final subsection is inspired by a very recent and elegant
article of C. Schnell cf. \cite{S}. The point in \cite{S}
is that Theorem \ref{bignu} can be obtained by combining
some of the main results established in the previous sections with
induction on the dimension on $X$. In this way one 
can avoid using the full force of the results we have in the general
orbifold context, provided that all the coefficients of the divisor $\Delta$ are equal to one. Nevertheless, the algebraicity criteria (Theorem 1.1) seems indispensable.

Notice however that, even for moduli problems, the treatment of multiple fibres requires the orbifold context. \medskip

\noindent We change slightly the notations: the orbifold
divisor $\Delta$ will be denoted here by $D= \sum D_i$, so as to indicate that
the pair $(X, D)$ is purely logarithmic. As before, $X$ is
non-singular 
and $D$ is a reduced divisor with simple normal crossings on $X$.
In what follows we will only be concerned with orbifold pairs $(X, D)$ of this type.

In this case the orbifold tangent bundle is the usual
logarithmic tangent bundle. This is a vector bundle on $X$,
sometimes denoted by
$T_X(-Log(D))$, but for the consistency's sake,
we will conserve the notation $T(X, D)$ here.

\noindent The logarithmic
tangent bundle $T(X, D)$ is closed under the Lie bracket induced from $T_X$., i.e. we have
\begin{equation}\label{mpvers01}
{\mathcal L}_{D}: \Lambda^2 T(X, D)\to T(X, D)
\end{equation}
given by the restriction of the Lie bracket of $X$ to the subsheaf
$T(X, D)$.
\smallskip

Let $\cF\subset T(X, D)$ be a coherent subsheaf. We have a
$\cO_X$-linear map
\begin{equation}\label{mpvers02}
{\mathcal L}_{D}^{\cF}:\Lambda^2 \cF\to T(X, D)/\cF
\end{equation}  
induced by ${\mathcal L}_{D}$.
\medskip

\noindent The following statements are particular cases 
 of Corollary \ref{flute} and of Theorem \ref{rco'}, respectively. In the purely logarithmic case, their proofs simplifies considerably, due to the fact that no adapted cover is needed.

\begin{lemma}\label{loglie}
  Let $\cF\subset T(X, D)$ be a coherent saturated subsheaf, such
  that the corresponding Lie bracket ${\mathcal L}_{D}^{\cF}$ vanishes identically. We denote by $\cF^s\subset T_X$ the saturation of
  $\cF$ in the tangent bundle of $X$.
  Then $\cF^s$ defines a holomorphic foliation.
\end{lemma}    

\begin{theorem}\label{logalgebraic} Let $\cF\subset T(X, D)$ be a
coherent subsheaf such that the corresponding Lie bracket ${\mathcal L}_{D}^{\cF}$ is identically zero. We assume moreover that
$\mu_{\alpha, min}(\cF)>0$, for some $\alpha\in \Mov(X)$.
Then the following are true.
\begin{enumerate}

\item[(1)] The saturation $\cF^s$ of $\cF$ in $T_X$ defines an algebraic foliation.
  \smallskip

\item[(2)] The restriction of $K_X+ D$ to the closure of the generic leaf
of the algebraic foliation $\cF^s\subset T_X$
is not pseudo-effective.   
\end{enumerate}  
\end{theorem}

\noindent 
As in the general case of an arbitrary orbifold divisor,
the conclusion of the point (2) of Theorem \ref{logalgebraic} 
means the following. There exists a birational map $p:X^\prime\to X$ such that
the support $D^\prime$ of $p^{-1}(D)$ has simple normal crossings,
together with a surjective map $f: X^\prime\to Z$ where $Z$ is a non-singular algebraic manifold, such that we have.

\noindent $\bullet$ The foliation induced by
$\cF^s$ on $X^\prime$ coincides generically with $\Ker (f)$, 

\noindent $\bullet$ The restriction
$\displaystyle K_{X^\prime}+ D^\prime|_{X_z^\prime}$ is not pseudo-effective, where $X^\prime_z$ is the fiber of $f$ at a generic point $z\in Z$.
\medskip

\begin{rem}\label{rweak} {\rm The conclusion here is considerably weaker than in the case where $D=0$ (rational connectedness being replaced by uniruledness). The analogous result in this generalised situation is established in \cite{Camp16}, after suitable equivalent definitions of rational connectedness in this context are given (based on negativity of the orbifold cotangent bundles, but without reference to `orbifold rational curves').}
\end{rem}

\noindent The next result is the particular case of Theorem \ref{rem1}, where $\Delta$ is reduced, which permits to give an extremely simple proof.

\begin{theorem}\label{logpseff} Let $(X,D)$ be a smooth projective connected purely logarithmic orbifold pair.
  Let $L$ be a pseudo-effective line bundle on $X$ such that there exists a sheaf embedding
  $$\displaystyle L\to \otimes^m
  \left(\Omega^1(X, D)\right)\otimes \big(K_X\otimes \cO_X(D)\big)^{\otimes p}$$ for some $m\geq 0, p>0$. Then $K_X+D$ is pseudo-effective.
\end{theorem}

\begin{proof} Assume by contradiction that $K_X+D$ is not pseudo effective. Let $\alpha\in \Mov(X)$ be such that $(K_X+D).\alpha<0$. Let $\cF\subset T(X,D)$ be a maximal destabilizing subsheaf, and $f:X\to Z$ be a fibration such that $\cF=Ker(df)$ generically (here we replace $(X,D)$ by a suitable birational smooth model in order to make $f$ regular, as explained in the bullets above). 
The generic orbifold fibre $(X_z, D_z)$ is thus smooth and $K_{X_z}+D_z$ not pseudo-effective.
\smallskip

\noindent

If $\dim(Z)=0$, then $\cF=T(X,D),$ and 
$0<\mu_{\alpha,min}(T(X,D))=-\mu_{\alpha,max}(\Omega(X,D)).$
In the following (in)equalities, we denote $\Omega(X,D):=\Omega, T$ is its dual, and $K:=K_X+D$: 

$0\leq L.\alpha\leq \mu_{\alpha, max}((\otimes^m\Omega\otimes K^p))=-m.\mu_{\alpha,min}(T)+p.K.\alpha\leq p.K.\alpha<0.$ Thus we have a contradiction.
\smallskip

If $\dim(Z)>0$, we have $0<dim(X_z)<dim(X)$, and we shall apply the preceding argument to $X_z$. Let $L_z:=L_{X_z}$; it is still pseudo-effective on $X_z$, and injects into $\otimes^m\Omega_z\otimes K_z^p$, with $\Omega_z,K_z$ the restrictions of $\Omega$ and $K_X+D$ to $X_z$ respectively. We have: $K_z=K_{X_z}+D_{X_z}$, and an exact sequence: $0\to \cO_{X_z}^{\oplus b}\to \Omega_z\to \Omega(X_z,D_z)\to 0$, with $b:=dim(Z)$. Tensoring $\otimes^m\Omega_z$ with $K_z^p$, and using the natural filtration on $\otimes ^m \Omega_z$ induced from the preceding exact sequence, we see that there is a nonzero map, for some $0\leq k\leq m$:
 $$\displaystyle L|_{X_z}\to \otimes^k
  \left(\Omega^1(X_z, D_z)\right)\otimes \big(K_{X_z}+D_z\big)^{\otimes p}$$
 By induction on $dim(X)$, we conclude that $K_{X_z}+D_z$ is pseudo-effective (when $k=0$, we need to use that $p\geq 1)$. This is a contradiction.
\end{proof}

\begin{corollary}\label{bigness} Let $(X,D)$ be a smooth projective connected purely logarithmic orbifold pair.
  Let $L$ be a line bundle on $X$, which admits an
  embedding $L\subset \otimes^m \Omega^1(X, D)$ for some $m>0$,
and such that the $\bQ$-bundle $\displaystyle \varepsilon(K_X+D)+ L$ is big for some rational number $\varepsilon \geq 0$. Then $K_X+D$ is big.
\end{corollary}

\begin{proof} Since $\displaystyle \varepsilon(K_X+D)+ L$ is big,
there exists an integer $q> 0$ such that the bundle $L_1:= L^q\otimes K_X\otimes \cO_X(D)$ is effective (the number $q$ depends on
  $(X, D), L$ and $\varepsilon$).

The hypothesis of \ref{bigness} shows that that $L_1$ admits an injection into
$\otimes^{mq} \Omega^1(X, D)\otimes  K_X\otimes \cO_X(D)$. If so, 
Theorem \ref{logpseff} implies that $K_X+ D$ is pseudo-effective. The corollary
then follows, since any quotient of $\otimes^{m} \Omega^1(X, D)$
has pseudo-effective determinant.
\end{proof}

%%%%%%%%%%%%%%%%%%%%%%%%%%%%%%%%%%%%%%%%%%%%%%%%%%%%%%%%%%%%%%%%%%%%%%%%%%%%%%%%%%%%%%%%%%%%%%%%%%%%%%%%%%%%%%%%%%%%%%%%%%%%%%%%%%%%%%%%%%%%%%%%%%%%%%%%%%%%%%%%%%%%%%%%%%%%%%%%%%%%%%%%%%%%%%%%%%%%%%%%%%%%%%%%%%%%%%%%

\end{document}